\documentclass{amsart}
\usepackage{graphicx,color}
\usepackage{amscd}
\usepackage[all]{xy}
\usepackage{enumerate}
\usepackage{hyperref}
\usepackage{subfigure}
\usepackage{amsmath,amssymb, amsthm}
\usepackage{fancybox}
\usepackage{amsfonts}

\theoremstyle{plain}
\newtheorem{theorem}{Theorem}[section]

\newtheorem{lemma}[theorem]{Lemma}
\newtheorem{proposition}[theorem]{Proposition}

\newtheorem*{claim}{Claim}

\theoremstyle{definition}

\newtheorem{example}[theorem]{Example}

\newtheorem{remark}[theorem]{Remark}



\newcommand{\CPb}{\overline{\mathbb{CP}}{}^{2}}
\newcommand{\CP}{{\mathbb{CP}}{}^{2}}

\newcommand{\R}{\mathbb{R}}

\newcommand{\Z}{\mathbb{Z}}
\newcommand{\D}{\displaystyle}
\newcommand{\ol}[1]{\overline{#1}}

\newcommand{\Pa}{\partial}
\newcommand{\smat}[1]{\left(\begin{smallmatrix}
#1
\end{smallmatrix}\right)}

\def\Ker{\operatorname{Ker}}

\def\Diff{\operatorname{Diff}}

\def\id{\operatorname{id}}
\def\Mod{\operatorname{Mod}}
\def\Crit{\operatorname{Crit}}

\def\grad{\operatorname{grad}}

\begin{document}

\title[On diagrams of simplified trisections and mapping class groups]
{On diagrams of simplified trisections and mapping class groups}

\author[K.~Hayano]{Kenta Hayano}
\address{
Department of Mathematics Faculty of Science and Technology, Keio University
Yagami Campus: 3-14-1 Hiyoshi, Kohoku-ku, Yokohama, 223-8522, Japan}
\email{k-hayano@math.keio.ac.jp}

\begin{abstract}

A simplified trisection is a trisection map on a $4$--manifold such that, in its critical value set, there is no double point and cusps only appear in triples on innermost fold circles. 
We give a necessary and sufficient condition for a $3$--tuple of systems of simple closed curves in a surface to be a diagram of a simplified trisection in terms of mapping class groups. 
As an application of this criterion, we show that trisections of spun $4$--manifolds due to Meier are diffeomorphic (as trisections) to simplified ones.
Baykur and Saeki recently gave an algorithmic construction of a simplified trisection from a directed broken Lefschetz fibration. 
We also give an algorithm to obtain a diagram of a simplified trisection derived from their construction. 

\end{abstract}

\maketitle

% ==========================================================================================================
\section{Introduction} 

A \emph{trisection}, due to Gay and Kirby \cite{GKtrisection}, is a decomposition of a $4$--manifold into three $4$-dimensional handlebodies, which can be considered as a $4$--dimensional counterpart of a Heegaard splitting of a $3$--manifold. 
Indeed, we can obtain a $4$--manifold with a trisection from a $3$--tuple of systems of simple closed curves in a closed surface, which is called a \emph{trisection diagram}, as we can obtain a $3$--manifold with a Heegaard splitting from a Heegaard diagram. 
On the other hand, we can also obtain a trisection of a $4$--manifold from a stable map to the plane with a specific configuration of the critical value set (see Figure~\ref{F:critv_trisection}), which we will call a \emph{trisection map} (or a trisection for simplicity). 

In studying smooth maps from $4$--manifolds to surfaces with stable and Lefschetz singularities, Baykur and Saeki \cite{BaykurSaeki} introduced \emph{simplified trisections}, which are trisection maps such that in their critical value sets, there are no double points and cusps only appear in triples on innermost fold circles (see Figure~\ref{F:critv simplified trisection} for the critical value set of a simplified trisection).  
They further gave an algorithm to obtain simplified trisections from directed broken Lefschetz fibrations, which implies the existence of a simplified trisection for an arbitrary $4$--manifold. 
In this paper, relying on the theory of mapping class groups of surfaces, we study trisection diagrams associated with simplified trisections. 

We first discuss when a $3$--tuple of (ordered) systems of simple closed curves is a diagram associated with a simplified trisection in terms of mapping class groups. 
Since the critical value set of a simplified trisection is nested circles with cusps (see Figure~\ref{F:critv simplified trisection}), we can take a monodromy along a loop between each consecutive pair of components in the critical value set. 
We will simultaneously give an algorithm to determine such a monodromy from curves in a diagram of a simplified trisection, and a necessary and sufficient condition for a given $3$--tuple of systems of curves to be a diagram associated with a simplified trisection (Theorem~\ref{T:descriptionSTviaMCG} and Lemma~\ref{T:mu is monodromy}).
Note that, strictly speaking, these results do not directly explain relation between simplified trisections and their diagrams: the curves $a_i,b_j,c_j$ in Theorem~\ref{T:descriptionSTviaMCG} are not curves in a trisection diagram, but vanishing cycles of indefinite folds of a simplified trisection (see the second paragraph of Section~\ref{S:descriptionSTviaMCG}). 
We will clarify relation between the vanishing cycles in Theorem~\ref{T:descriptionSTviaMCG} and curves in a diagram of a simplified trisection (Proposition~\ref{T:relation vc diagram}). 
As an application of the results, we will show that trisections of spun $4$--manifolds constructed by Meier \cite{Meier} are diffeomorphic to simplified trisections (Theorem~\ref{T:Meier's ex are simplified}). 
We will also classify simplified trisections with genus $2$ using Theorem~\ref{T:descriptionSTviaMCG} (Theorem~\ref{T:classificationgenus2ST}). 
This classification can also be obtained as merely a corollary of the classification of genus--$2$ trisections given in \cite{MZgenus2}, in which the authors relied on deep results on genus--$2$ Heegaard splittings of $S^3$, while we will reduce the classification to linear algebraic problems, which is easy to solve. 

By analyzing how vanishing cycles are changed in algorithmic construction of simplified trisections from broken Lefschetz fibrations given in \cite{BaykurSaeki}, we next give an algorithm to obtain trisection diagrams from vanishing cycles of broken Lefschetz fibrations. 
In the analysis of vanishing cycles, we will be faced with problems concerning how parallel transports are affected by homotopies of stable maps called $\mathrm{R2}$--moves, which change the critical value sets like the Reidemeister move of type II.  
We will solve such problems by making use of the results by the author \cite{HayanoR2}, which completely describe the effect of $\mathrm{R2}$--moves on parallel transports (and thus that on vanishing cycles) in terms of mapping class groups. 
Lastly we will apply the algorithm to genus--$1$ simplified broken Lefschetz fibrations, resulting in diagrams associated with simplified trisections on $S^4$, the connected sum of $S^1\times S^3$ and an $S^2$--bundle over $S^2$, and manifolds $L_n$ and $L_n'$ due to Pao~\cite{Pao} (Example~\ref{Ex:simplified trisection genus-1SBLF}). 
Using the diagram obtained there, we will verify that a simplified $(3,1)$--trisection on $S^4$ obtained from a genus--$1$ simplified broken Lefschetz fibration is diffeomorphic to the stabilization of the $(0,0)$--trisection of $S^4$ (Figure~\ref{F:standard diagram S4} explicitly gives sequences of handle-slides between diagrams corresponding to the two trisections). 

\section{Preliminaries}

Throughout the paper, we will assume that any manifold is smooth, connected, compact and oriented unless otherwise noted. 
We will describe the image of definite (resp.~indefinite) folds by red (resp.~black) curves with co-orientations as shown in Figures~\ref{F:definite fold} and \ref{F:indefinite fold}.
As shown in Figure~\ref{F:Lefschetz}, we will describe the image of a Lefschetz singularity by a cross. 
Note that these conventions are the same as those in \cite{BaykurSaeki}. 
\begin{figure}[htbp]
\centering
\subfigure[Definite folds]{\includegraphics[width=27mm]{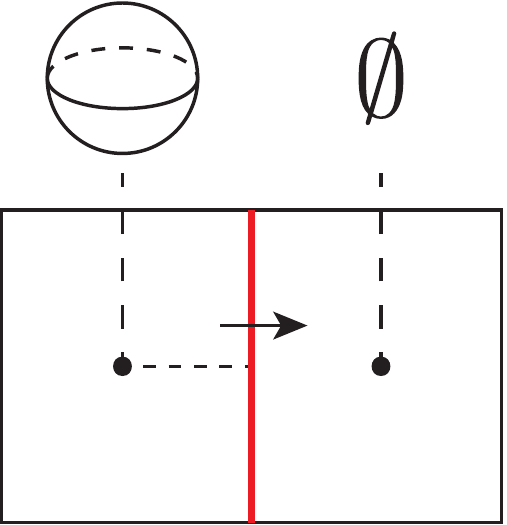}
\label{F:definite fold}}
\subfigure[Indefinite folds]{\includegraphics[width=27mm]{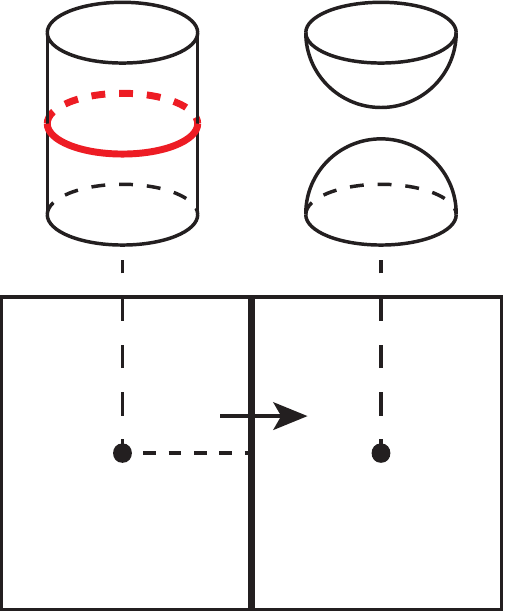}
\label{F:indefinite fold}}
\subfigure[A Lefschetz singularity]{\includegraphics[width=27mm]{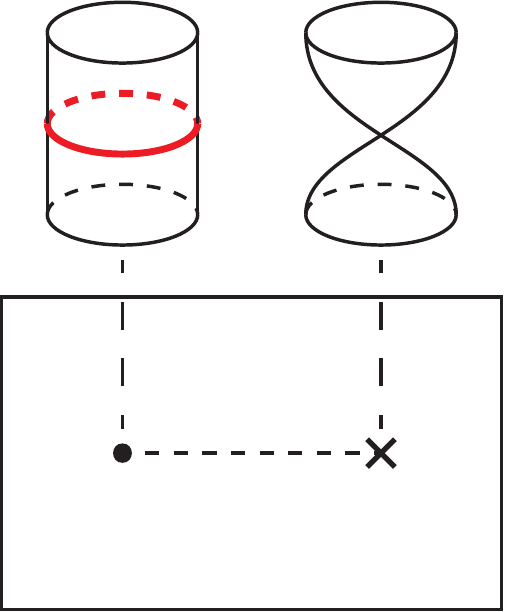}
\label{F:Lefschetz}}
\caption{Fibers around critical values.}
\label{F:critical points}
\end{figure}

\subsection{Mapping class groups and monodromies along indefinite folds}\label{S:monodromy indeffold}

Let $\Sigma$ be a surface, $V_1,\ldots,V_k\subset \Sigma$ discrete subsets, and $c_1,\ldots,c_l$ simple closed curves.  
We denote the set of orientation preserving diffeomorphisms fixing $V_1,\ldots,V_k$ setwise by $\Diff^+(\Sigma;V_1,\ldots,V_k)$, and we put 
\begin{align*}
&\Mod(\Sigma;V_1,\ldots,V_k)(c_1,\ldots,c_l) =\left\{\left.[\varphi]\in\pi_0\left(\Diff^+(\Sigma;V_1,\ldots,V_k)\right)\right|\varphi(c_i)=c_i\hspace{.3em}(i=1,\ldots,l)\right\}. 
\end{align*}
We endow it with a group structure by compositions of representatives. 
Let $\Sigma_{c_i}$ be a surface obtained by applying surgery along $\Sigma$ (i.e.~attaching two disks to $\Sigma\setminus \nu (c_i)$, where $\nu (c_i)\subset \Sigma$ is a tubular neighborhood of $c_i$), and $v_0,v_0'\in \Sigma_{c_i}$ the centers of the disks. 
We define a \emph{(pointed) surgery homomorphism}
\begin{align*}
\Phi_{c_i}^\ast : & \Mod(\Sigma;V_1,\ldots,V_k)(c_1,\ldots,c_l)\to \Mod(\Sigma_{c_i};\{v_0,v_0'\},V_1,\ldots,V_k)(c_1,\ldots,,c_{i-1},c_{i+1},\ldots,c_l)
\end{align*}
as follows.
For an element $\varphi\in \Diff^+(\Sigma;V_1,\ldots,V_k)$, we first modify $\varphi$ by an isotopy so that it preserves $\nu(c_i)$. 
Let $\tilde{\varphi}:\Sigma_{c_i}\to \Sigma_{c_i}$ be an extension of the diffeomorphism $\varphi|_{\Sigma\setminus \nu(c_i)}$ preserving the set $\{v_0,v_0'\}$. 
We then put $\Phi_{c_i}^\ast([\varphi]) = [\tilde{\varphi}]$. 
It is known that the map $\Phi_{c_i}^\ast$ is well-defined and the kernel of $\Phi_{c_i}^\ast$ is generated by the Dehn twist $t_{c_i}$ (for the proof, see \cite[Lemma 3.1]{BaykurHayanoBLF}\footnote{The authors merely dealt with the case $k=0$ and $l=1$ in \cite{BaykurHayanoBLF}, yet the proof there works for general cases.}, for example). 
We put $\Phi_{c_i}=F_{v_0,v_0'}\circ \Phi_{c_i}^\ast$, which is also called a \emph{surgery homomorphism}, where $F_{v_0,v_0'}$ is the forgetting map. 

Let $f:X\to \Sigma$ be a smooth map on a $4$--manifold to a surface and $\mathcal{S}\subset \Crit(f)$ be a circle consisting of indefinite folds. 
Assume that the restriction $f|_{\mathcal{S}}$ is embedding and the complement $\nu(f(\mathcal{S}))\setminus f(\mathcal{S})$, where $\nu(f(\mathcal{S}))$ is a tubular neighborhood of $f(\mathcal{S})$, does not contain any critical values of $f$. 
The complement $\nu(f(\mathcal{S}))\setminus f(\mathcal{S})$ has two components. 
We take two points $p_0, q_0$ from each of the components. 
Let $\alpha$ and $\beta$ be loops in $\nu(f(\mathcal{S}))\setminus f(\mathcal{S})$ based at $p_0$ and $q_0$, respectively, such that these curves represent the same element in $H_1(\nu(f(\mathcal{S}));\Z)$. 
We take an oriented simple path $\gamma\subset\nu(f(\mathcal{S}))$ from $p_0$ to $q_0$ intersecting with $f(\mathcal{S})$ on one point transversely. 
Assume that the orientation of $\gamma$ coincides with the co-orientation of $f(\mathcal{S})$ at the intersection. 
The arc $\gamma$ gives rise to a vanishing cycle $c\subset \Sigma=f^{-1}(p_0)$ and an identification of $f^{-1}(q_0)$ with $\Sigma_c$. 
It is known that the monodromy $\mu \in \Mod(\Sigma)$ along $\alpha$ is contained in $\Mod(\Sigma)(c)$ and, under the identification above, the monodromy along $\beta$ is equal to $\Phi_c(\mu)$ (see \cite{Baykur_PJM,BaykurHayanoBLF}).

\subsection{Trisections of $4$--manifolds}\label{S:review trisection}

Let $X$ be a closed $4$--manifold. 
A decomposition $X=X_1\cup X_2\cup X_3$ is called a \emph{$(g,k)$--trisection} if

\begin{itemize}

\item 
for each $i=1,2,3$, there is a diffeomorphism $\phi_i:X_i \to \natural^k(S^1\times D^3)$, and

\item 
for each $i=1,2,3$, taking indices mod $3$, $\phi_i(X_i\cap X_{i+1})=Y_{k,g}^-$ and $\phi_i(X_i\cap X_{i-1})=Y_{k,g}^+$, where $\Pa \left(\natural^k (S^1\times D^3)\right) = \sharp^k(S^1\times S^2) = Y_{k,g}^+\cup Y_{k,g}^-$ is a Heegaard splitting of $\sharp^k(S^1\times S^2)$ obtained by stabilizing the standard genus--$k$ Heegaard splitting $g-k$ times.

\end{itemize}

\noindent
A stable map $f:X \to \R^2$ is called a \emph{$(g,k)$--trisection map} if its critical value set is as shown in Figure~\ref{F:critv_trisection}, where each of the three white boxes, called a \emph{Cerf box} of $f$, consists of a Cerf graphic without cusps and definite folds (i.e.~it consists of indefinite fold images with transverse double points but without ``radial tangencies''). 
A trisection map $f$ is said to be \emph{simplified} if its Cerf boxes do not contain double points. 
Note that, by the definition of a simplified trisection, cusps of it only appear in triples on innermost fold circles (see Figure~\ref{F:critv simplified trisection}).
\begin{figure}[htbp]
\includegraphics[width=65mm]{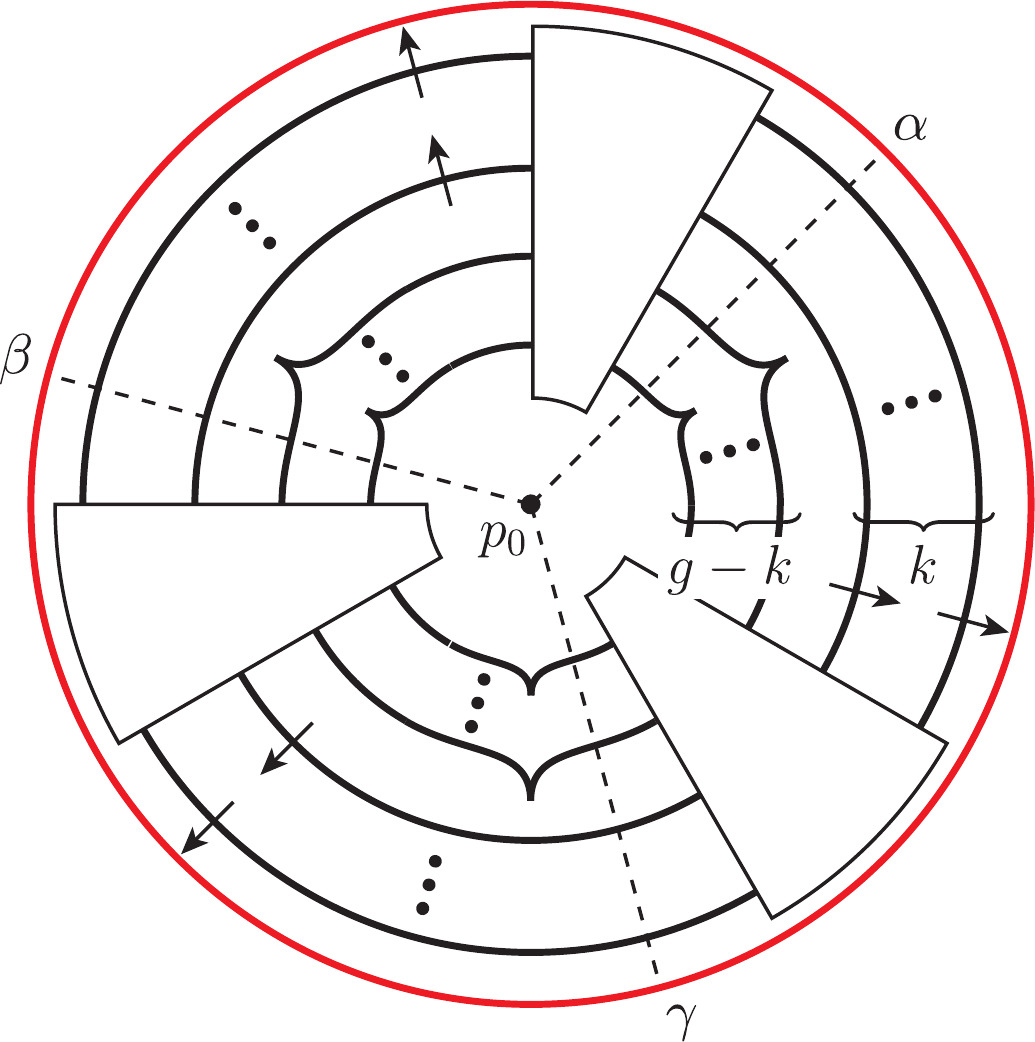}
\caption{The critical value set of a trisection map.}
\label{F:critv_trisection}
\end{figure}
Let $\alpha,\beta, \gamma$ be $g$--tuples of simple closed curves in $\Sigma_g$. 
A tuple $(\Sigma_g;\alpha,\beta,\gamma)$ is called a \emph{$(g,k)$--trisection diagram} if each of the tuples $(\Sigma_g;\alpha,\beta), (\Sigma_g;\beta,\gamma)$ and $(\Sigma_g;\gamma, \alpha)$ is a genus--$g$ Heegaard diagram of $\sharp^k(S^1\times S^2)$. 
In this paper we will mainly focus on trisection maps and diagrams. 
For this reason, we will sometimes call trisection maps just trisections for simplicity. 

From a $(g,k)$--trisection map $f:X\to \R^2$, we can obtain a trisection of $X$ and a trisection diagram as follows. 
The three dotted segments in Figure~\ref{F:critv_trisection} decomposes the image of $f$ into three regions $D_1, D_2$ and $D_3$.
The decomposition $X=X_1\cup X_2 \cup X_3$ is a $(g,k)$--trisection of $X$, where $X_i=f^{-1}(D_i)$. 
Furthermore, by taking vanishing cycles of $f$ with respect to the three dotted segments in Figure~\ref{F:critv_trisection}, we can obtain a three $g$--tuples $\alpha=(\alpha_1,\ldots,\alpha_g), \beta=(\beta_1,\ldots,\beta_g), \gamma=(\gamma_1,\ldots,\gamma_g)$ of simple closed curves in $f^{-1}(p_0)\cong \Sigma_g$. 
The tuple $(\Sigma_g;\alpha,\beta,\gamma)$ is a $(g,k)$--trisection diagram. 

Conversely, for a $(g,k)$--trisection diagram $(\Sigma_g;\alpha,\beta,\gamma)$ (or a $(g,k)$--trisection of $X$), there exists a $(g,k)$--trisection map $f:X\to \R^2$ such that the corresponding diagram obtained in the procedure above is $(\Sigma_g;\alpha,\beta,\gamma)$. 
Note that the diffeomorphism type of $X$ is uniquely determined from the diagram $(\Sigma_g;\alpha,\beta,\gamma)$, while a trisection map $f$ is not uniquely determined (there are many choices of Cerf boxes, for example). 

Two trisections $X=X_1\cup X_2\cup X_3$ and $X'=X_1'\cup X_2'\cup X_3'$ are said to be \emph{diffeomorphic} if there exists an orientation preserving diffeomorphism $\Phi:X\to X'$ sending $X_i$ to $X_i'$. 
It is known that two trisections of the same manifold become diffeomorphic after stabilizing each trisection several times (for the definition of stabilization and the proof of this statement, see \cite{GKtrisection}).
Since any two handle decompositions of a genus--$g$ handlebody with one $0$--handle and $g$ $1$--handles are related by handle-slides, two trisections are diffeomorphic if and only if the corresponding diagrams are related by orientation preserving self-diffeomorphisms of $\Sigma_g$ and handle-slides, which are slides of $\alpha$-- (resp.~$\beta$-- and $\gamma$--) curves over $\alpha$-- (resp.~$\beta$-- and $\gamma$--) curves.

\section{Description of simplified trisections via mapping class groups}\label{S:descriptionSTviaMCG}

As we briefly reviewed in Section~\ref{S:review trisection}, three $g$--tuples $\alpha,\beta,\gamma$ of simple closed curves in $\Sigma_g$ is a $(g,k)$--trisection diagram if and only if each two of the tuples is a genus--$g$ Heegaard diagram of $\sharp^k(S^1\times S^2)$. 
In this section, after observing relation between diagrams and vanishing cycles of simplified trisections (Proposition~\ref{T:relation vc diagram}), we will give a necessary and sufficient condition for systems of simple closed curves to be vanishing cycles of a \emph{simplified} $(g,k)$--trisection in terms of mapping class groups. 
Making use of this condition, we will then show that trisections of spun $4$--manifolds constructed by Meier \cite{Meier} are diffeomorphic to simplified ones (Theorem~\ref{T:Meier's ex are simplified}). 
We will also give an elementary proof (without relying on the more general result in \cite{MZgenus2}) of the classification of simplified trisections with genus $2$ (Theorem~\ref{T:classificationgenus2ST}). 

Let $f:X\to \R^2$ be a simplified $(g,k)$--trisection. 
We take reference paths in $\R^2$ as shown by dotted segments in Figure~\ref{F:critv simplified trisection}. 
We denote the fiber on the center by $\Sigma$. 
Let $a_i,b_i,c_i\subset \Sigma$ be vanishing cycles of the $i$--th innermost indefinite fold circle associated with the reference paths in Figure~\ref{F:critv simplified trisection}. 
In what follows, the vanishing cycles are assumed to be in general position. 
Note that $a_i, b_i$ and $c_i$ are unique up to isotopies and handle-slides over $a_1,\ldots,a_{i-1}$. 
Following the observation in \cite[Remark 7.2]{BaykurSaeki}, we call such handle-slides \emph{upper-triangular handle-slides}. 
\begin{figure}[htbp]
	\includegraphics[width=80mm]{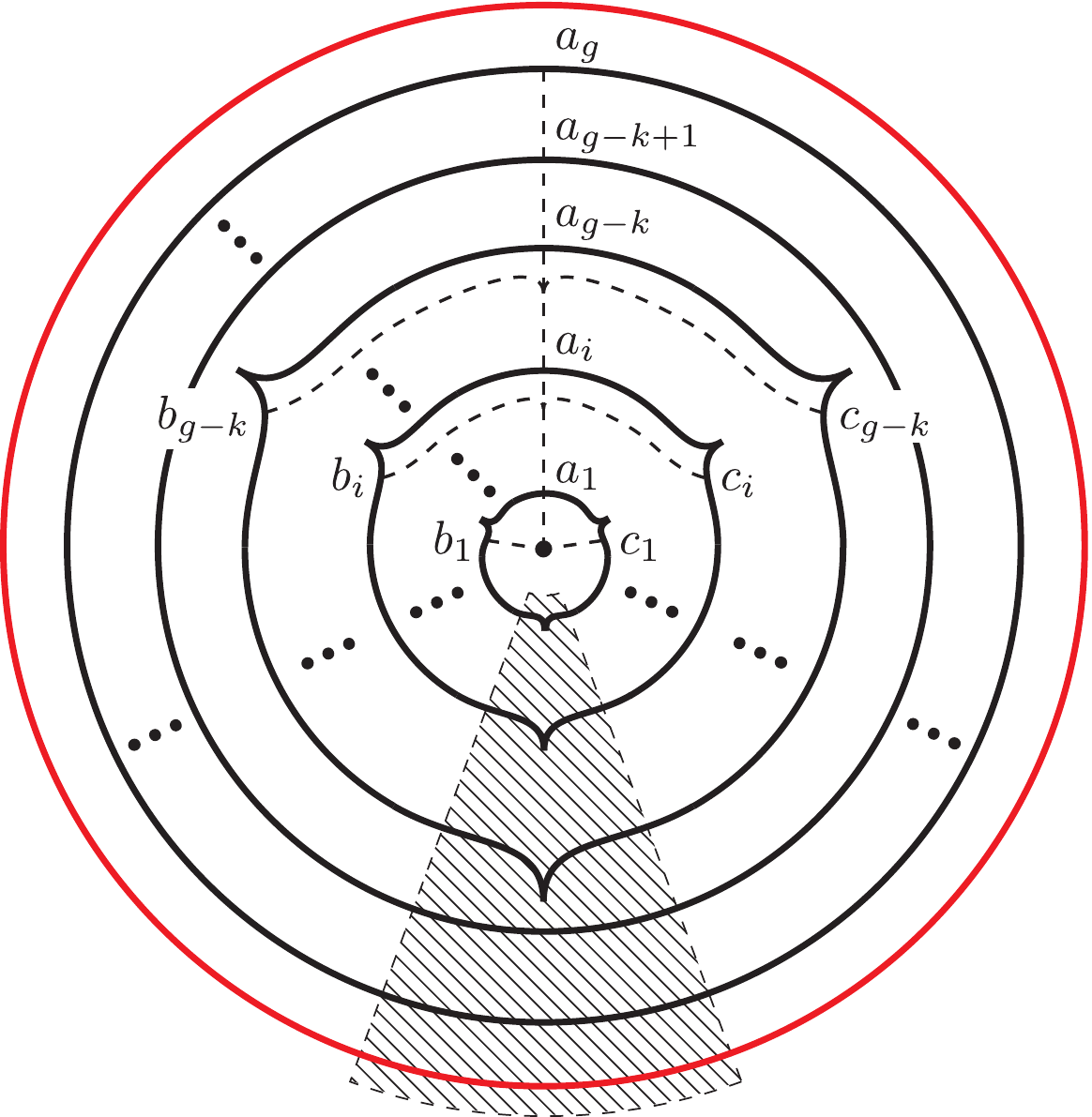}
	\caption{Reference paths for a simplified $(g,k)$--trisection.}
	\label{F:critv simplified trisection}
\end{figure}

\begin{proposition}\label{T:relation vc diagram}

We can obtain a trisection diagram $(\Sigma;\alpha,\beta,\gamma)$ associated with $f$ by taking $\alpha_i,\beta_i,\gamma_i$ as follows: 

\begin{itemize}
	
	\item 
	for each $i=1,\ldots,g$, $\alpha_i = a_i$, 
	
	\item 
	for each $j=1,\ldots,g-k$, $\beta_j$ (resp.~$\gamma_j$) is a curve obtained by applying upper-triangular handle-slides to $b_j$ (resp.~$c_j$) so that the resulting curve is disjoint from the curves $b_1,\ldots,b_{j-1}$ (resp.~$c_1,\ldots,c_{j-1}$), 
	
	\item 
	for each $l=g-k+1,\ldots,g$, $\beta_l$ (resp.~$\gamma_l$) is a curve obtained by applying upper-triangular handle-slides to $a_l$ so that the resulting curve is disjoint from $b_1,\ldots,b_{l-1}$ (resp.~$c_1,\ldots,c_{l-1}$), 
	
\end{itemize}
Conversely, let $(\Sigma;\alpha',\beta',\gamma')$ be a trisection diagram associated with a simplified $(g,k)$--trisection $f'$ and $a_i',b_j',c_j'$ ($i=1,\ldots,g$, $j=1,\ldots,g-k$) vanishing cycles associated with the reference paths in Figure~\ref{F:critv simplified trisection}. 
Then the followings hold up to upper-triangular handle-slides: 
\begin{itemize}
	
	\item 
	for each $i=1,\ldots,g$, $a_i' = \alpha_i'$, 
	
	\item 
	for each $j=1,\ldots,g-k$, $b_j'$ (resp.~$c_j'$) is equal to a curve obtained by applying upper-triangular handle-slides to $\beta_j'$ (resp.~$\gamma_j'$) so that the resulting curve is disjoint from the curves $\beta_1',\ldots,\beta_{j-1}'$ (resp.~$\gamma_1',\ldots,\gamma_{j-1}'$),

	\item
	for each $l=g-k+1,\ldots,g$, $\alpha_l' = \beta_l'=\gamma_l'$.  
\end{itemize}
\end{proposition}

\noindent
Proposition~\ref{T:relation vc diagram} can be deduced immediately from the following lemma:

\begin{lemma}\label{T:vc consecutivecusps}
	
	Let $d_1,d_2,d_3,d_4$ be vanishing cycles associated with the corresponding reference paths around two consecutive cusps as shown in Figure~\ref{F:refpath consecutivecusp1}. 
	The curve $d_4$ is isotopic to $t_{t_{d_1}(d_2)}^{-1}(d_3)$. 
	In particular $d_3$ and $d_4$ are isotopic if $d_3$ is disjoint from $d_2$. 
	\begin{figure}[htbp]
		\subfigure[]{\includegraphics[width=22mm]{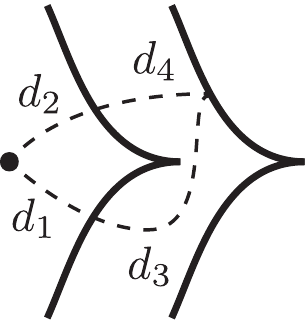}
			\label{F:refpath consecutivecusp1}}\hspace{2.5em}
		\subfigure[]{\includegraphics[width=22mm]{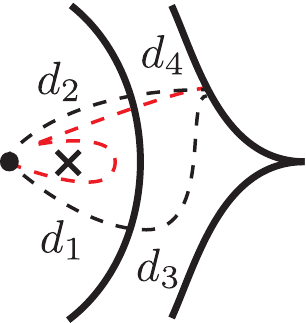}
			\label{F:refpath consecutivecusp2}}
		\caption{Reference paths around two consecutive cusps.}
		\label{F:refpath consecutivecusp}
	\end{figure}
\end{lemma}

\begin{remark}

Since $d_1$ and $d_2$ intersect on one point, we can always make $d_3$ disjoint from $d_2$ by applying handle-slides over $d_1$. 

\end{remark}

\begin{proof}[Proof of Lemma~\ref{T:vc consecutivecusps}]
	Applying unsink to the left cusp in Figure~\ref{F:refpath consecutivecusp1}, we obtain a new map whose critical value set is shown in Figure~\ref{F:refpath consecutivecusp2}. 
	The curve $d_3$ is isotopic to a vanishing cycle associated with the red reference path in Figure~\ref{F:refpath consecutivecusp2}. 
	Thus, $d_3$ is sent to $d_4$ by the Dehn twist along the curve ${t_{d_1}(d_2)}$. 
\end{proof}

We are now ready for giving a criterion for a system of simple closed curves to be a diagram of a simplified trisection, in terms of the corresponding vanishing cycles. 

\begin{theorem}\label{T:descriptionSTviaMCG}
 
Vanishing cycles $a_1,\ldots,a_g,b_1,\ldots,b_{g-k},c_1,\ldots,c_{g-k}$ of a simplified $(g,k)$--trisection taken as above satisfy the following conditions: 

\begin{enumerate}

\item 
the curves $b_i$ and $c_i$ intersect with $a_i$ on one point for each $i=2,\ldots,g-k$,

\item 
for any $i\in \{0,\ldots,g-k-1\}$, $c_{i+1}' = \mu_i^{-1}(c_{i+1})$ intersects $b_{i+1}$ on one point, where $\Sigma_{a_1,\ldots,a_i}=(\Sigma_{a_1,\ldots,a_{i-1}})_{a_i}$ and $\mu_i\in \Mod(\Sigma_{a_1,\ldots,a_i})$ is inductively defined as
\[
\mu_i = \begin{cases}
[\id_{\Sigma}] & (i=0)\\
\Phi_{a_i}(t_{t_{c_{i}}a_{i}}\circ \mu_{i-1}\circ t_{t_{b_{i}}c_{i}'}\circ t_{t_{a_{i}}b_{i}}) & (i>0). 
\end{cases}
\]
(Note that $t_{t_{c_{i}}a_{i}}\circ \mu_{i-1}\circ t_{t_{b_{i}}c_{i}'}\circ t_{t_{a_{i}}b_{i}}$ preserves $a_i$ if $c_i'$ intersects $b_i$ on one point. 
In particular the condition for $i$ makes sense only if that for $i-1$ holds.) 

\item 
for any $j\in \{g-k,\ldots,g-1\}$, $\Phi_{a_j}\circ \cdots \circ \Phi_{a_{g-k+1}}(\mu_{g-k})$ preserves $a_{j+1}$. 

\end{enumerate}

\noindent
Conversely, for three sequences of curves $\{a_1,\ldots,a_g\},\{b_1,\ldots,b_{g-k}\},\{c_1,\ldots,c_{g-k}\}$ satisfying the conditions above (such that $b_i$ and $c_i$ are disjoint from $a_1,\ldots,a_{i-1}$), there exists a simplified $(g,k)$--trisection $f':X'\to \R^2$ whose vanishing cycles associated with the reference paths in Figure~\ref{F:critv simplified trisection} are $a_1,\ldots,a_g,b_1,\ldots,b_{g-k},c_1,\ldots,c_{g-k}$. 

\end{theorem}

\noindent
We need the following lemma to prove Theorem~\ref{T:descriptionSTviaMCG}. 

\begin{lemma}\label{T:mu is monodromy}
	
	The mapping class $\mu_i$ defined in Theorem~\ref{T:descriptionSTviaMCG} is the monodromy along a loop (with counterclockwise orientation) obtained by pushing the $i$--th innermost circle to the inside of it.
	
\end{lemma}

\begin{proof}
	We prove the lemma by induction on $i$. 
	The statement is obvious for $i=0$. 
	Assume that the lemma holds for $i=j-1$, where $j\leq g-k-1$. 
	The curve $c_i'=\mu_{i-1}^{-1}(c_i)$ is a vanishing cycle of $f$ associated with the reference path shown in Figure~\ref{F:refpaths simplifiedtrisection1}. 
	\begin{figure}[htbp]
		\subfigure[]{\includegraphics[width=40mm]{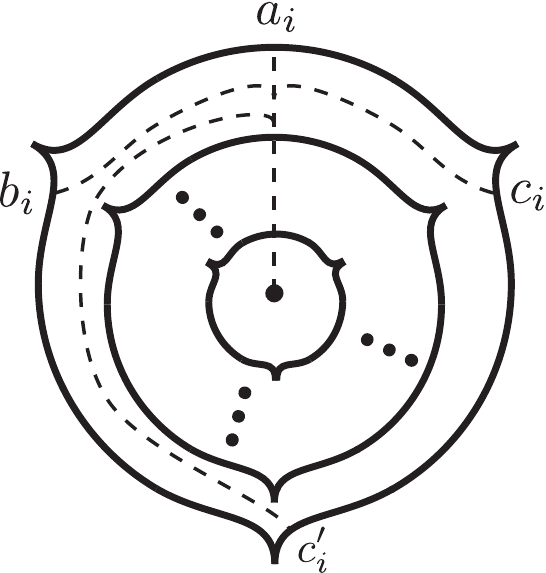}
			\label{F:refpaths simplifiedtrisection1}}
		\subfigure[]{\includegraphics[width=40mm]{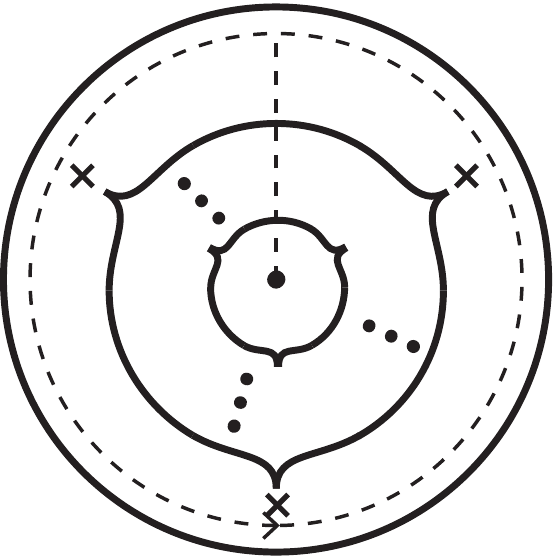}
			\label{F:refpaths simplifiedtrisection2}}
		\caption{Reference paths for \subref{F:refpaths simplifiedtrisection1} vanishing cycles $a_i,b_i,c_i,c_i'$ and \subref{F:refpaths simplifiedtrisection2} a monodromy $t_{t_{c_{i}}a_{i}}\circ \mu_{i-1}\circ t_{t_{b_{i}}c_{i}'}\circ t_{t_{a_{i}}b_{i}}$.}
		\label{F:refpaths simplifiedtrisection}
	\end{figure}
	We apply unsinks to the three cusps on the $i$--th innermost fold circles. 
	The critical value set of the resulting map is shown in Figure~\ref{F:refpaths simplifiedtrisection2}. 
	Vanishing cycles of the three Lefschetz singularities are respectively equal to $t_{a_i}(b_i), t_{b_i}(c_i')$ and $t_{c_i}(a_i)$. 
	Thus the monodromy along the dotted curve in Figure~\ref{F:refpaths simplifiedtrisection2} is equal to $t_{t_{c_{i}}a_{i}}\circ \mu_{i-1}\circ t_{t_{b_{i}}c_{i}'}\circ t_{t_{a_{i}}b_{i}}$. 
	The statement then follows from the observation in Section~\ref{S:monodromy indeffold}. 
	Lastly, the statement for $\mu_j$ with $j\geq g-k$ also follows from the observation in Section~\ref{S:monodromy indeffold}, together with the induction hypothesis. 
\end{proof}

\begin{proof}[Proof of Theorem~\ref{T:descriptionSTviaMCG}]
Let $a_1,\ldots,a_g,b_1,\ldots,b_{g-k},c_1,\ldots,c_{g-k}$ be vanishing cycles of a simplified $(g,k)$--trisection taken as above. 
The condition (1) holds since two vanishing cycles of indefinite folds around a cusp intersect on one point, while the conditions (2) and (3) immediately follow from Lemma~\ref{T:mu is monodromy}.  

In order to prove the latter part of Theorem~\ref{T:descriptionSTviaMCG}, suppose that three systems of curves $(a_1,\ldots,a_g),\allowbreak(b_1,\ldots,b_{g-k}),(c_1,\ldots,c_{g-k})$ satisfy the conditions in the theorem. 
The condition (1) guarantees existence of a map $f_1:X_1\to \R^2$, where $X_1$ is a $4$--manifold with boundary, such that the critical value set of $f_1$ is the same as that in the complement of the shaded region in Figure~\ref{F:critv simplified trisection} and it has the desired vanishing cycles $a_1,\ldots,a_g,b_1,\ldots,b_{g-k},c_1\ldots,c_{g-k}$. 
We can inductively see that the condition (2) (together with Lemma~\ref{T:mu is monodromy}) guarantees that we can attach $g-k$ cusps in the shaded region to $f_1$, and the condition (3) (together with Lemma~\ref{T:mu is monodromy}) further guarantees that the resulting map can be extended to the all of the shaded region.  
\end{proof}

For applications of Theorem~\ref{T:descriptionSTviaMCG}, we need to describe the monodromy $\mu_{i+1}$ of a simplified trisection as a product of Dehn twists under the assumption that the inner monodromy $\mu_{i}$ is trivial:

\begin{lemma}\label{T:monodromyalongtriangle}

Let $f:X\to \R^2$ be a simplified trisection.
Suppose that the monodromy $\mu_{i}$ obtained as in Theorem~\ref{T:descriptionSTviaMCG} is trivial.
We denote $\Sigma_{a_1,\ldots,a_{i}}$ by $\Sigma$ and $a_{i+1}$ (resp.~$b_{i+1},c_{i+1}$) by $a$ (resp.~$b,c$) for simplicity. 

\begin{enumerate}

\item 
A regular neighborhood of $a\cup b\cup c$ is a genus--$1$ surface with three boundary components. 

\item 
The monodromy $\mu_{i+1}$ is $t_{\delta_1}^2t_{\delta_3}^2t_{\delta_2}^{-1}$ if we can take orientations of $a,b$ and $c$ so that the algebraic intersections $a\cdot b$, $b\cdot c$ and $c\cdot a$ are all equal to $1$ (Figure~\ref{F:vc_triangle2}), and is equal to $t_{\delta_1}^{-2}t_{\delta_3}^{-2}t_{\delta_2}$ otherwise (Figure~\ref{F:vc_triangle1}). 

\end{enumerate}
\begin{figure}[htbp]
\subfigure[]{\includegraphics[width=37mm]{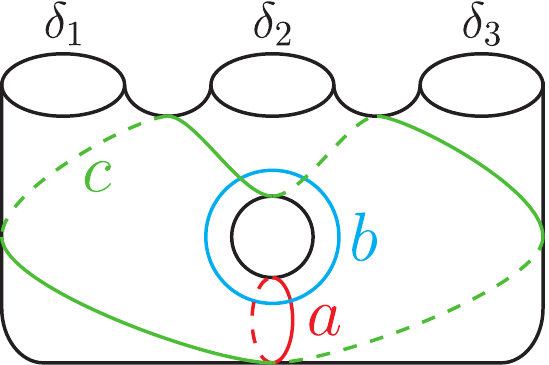}
\label{F:vc_triangle2}}
\subfigure[]{\includegraphics[width=37mm]{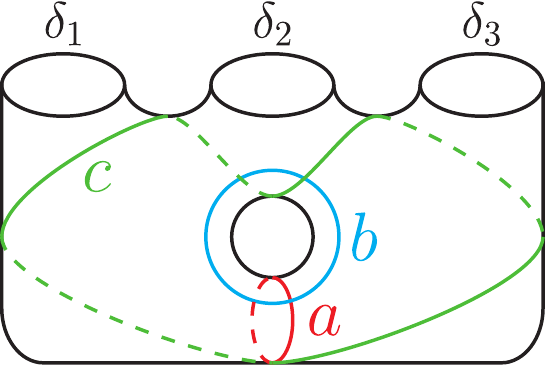}
\label{F:vc_triangle1}}
\caption{Two possibilities of configurations of vanishing cycles of $f$. }
\label{F:monodromytriangle}
\end{figure}
\end{lemma}

\begin{proof}
The first statement immediately follows from Theorem~\ref{T:descriptionSTviaMCG} and the assumption (note that each two of the curves $a,b,c$ intersect on one point). 
Since the monodromy $\mu_{i+1}$ becomes the inverse of it when we change the orientation of $X$ (and thus that of $\Sigma$), it is enough to show the statement under the assumption that we can take orientations of $a,b$ and $c$ so that the algebraic intersections $a\cdot b$, $b\cdot c$ and $c\cdot a$ are all equal to $1$ (i.e.~$a,b$ and $c$ are as shown in Figure~\ref{F:vc_triangle2}). 
In what follows, for (an isotopy class of) a simple closed curve $d$ in $\Sigma$, we also denote the Dehn twist along it by $d$ (which is contained in $\Mod(\Sigma;\Pa \Sigma)$). 
For an element $\varphi\in \Mod(\Sigma;\Pa \Sigma)$, let ${}_\varphi(d)$ be the isotopy class of a simple closed curve in $\Sigma$ represented by the image of $d$ by a representative of $\varphi$. 

We take simple closed curves $x$ and $y$ as shown in Figure~\ref{F:scc_neigh_abc}. 
\begin{figure}[htbp]
\includegraphics[width=35mm]{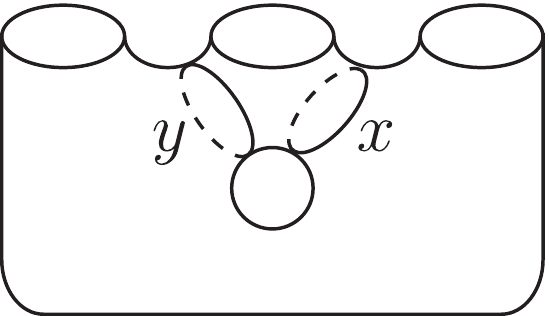}
\caption{Simple closed curves in a neighborhood of $a\cup b\cup c$.}
\label{F:scc_neigh_abc}
\end{figure}
By Lemma~\ref{T:mu is monodromy} and the assumption, the monodromy $\mu_{i+1}$ is equal to $\Phi_a({}_c(a){}_b(c){}_a(b))$. 
Since the curve $c$ is equal to ${}_{x y \overline{a}}(b)$, we can calculate ${}_c(a){}_b(c){}_a(b)$ as follows (in the following calculation, the only underlined part in each line is changed when proceeding to the next line):
{\allowdisplaybreaks
\begin{align*}
{}_c(a){}_b(c){}_a(b) =& \underline{c}a\underline{\overline{c}}\cdot b\underline{c}\overline{b}\cdot ab\overline{a} \\
=& (xy\ol{a}b\underline{\ol{x}\ol{y}a)\cdot a \cdot (xy\ol{a}}\ol{b}\ol{x}\ol{y}a)\cdot b \cdot (xy\ol{a}\underline{b\ol{x}\ol{y}}\underline{a) \cdot \ol{b} \cdot ab\ol{a}} \\
=& xy\underline{\ol{a}b\cdot (a)\cdot\ol{b}}\ol{x}\ol{y}a b  \underline{xy\ol{a}\cdot( {}_b(\ol{x}){}_b(\ol{y})}\hspace{.1em}\underline{b)(a^2b\ol{a}^2)} \\
=& xy\cdot (\ol{a}^2ba)\cdot \ol{x}\ol{y}\underline{ab \cdot (\ol{a}x \ol{b}}y\cdot {}_b(\ol{x}))\cdot (\Delta_{a,b}\ol{a}^4) \\
=& xy\ol{a}^2b a \ol{x}\ol{y}\cdot (\ol{b}a\ol{x}b \underline{x) \cdot y {}_b(\ol{x})} \Delta_{a,b}\ol{a}^4 \\
=& xy\ol{a}^2b a \ol{x}\underline{\ol{y}\ol{b}a\ol{x}b \cdot (y \ol{b}}x)\cdot\Delta_{a,b}\ol{a}^4 \\
=& xy\ol{a}^2  b a\ol{x}\cdot (\ol{b}\ol{y}\underline{\ol{b}ab}\hspace{.1em} \underline{\ol{b} \ol{x}b}y) \cdot x\Delta_{a,b}\ol{a}^4 \\
=& xy\ol{a}^2  b a\ol{x} \underline{\ol{b}\ol{y}\cdot (aba}\ol{a}^2x\ol{b} \ol{x})\cdot y x\Delta_{a,b}\ol{a}^4 \\
=& xy\ol{a}^2  b a\ol{x}\cdot (y\ol{b}\ol{y} \underline{a}b)\cdot \ol{a}^2x\ol{b}\ol{x}yx\Delta_{a,b}\ol{a}^4 \\
=& xy\ol{a}^2  \underline{b a\ol{x} y\ol{b}\ol{y}\cdot (x\ol{a}\ol{b}}\hspace{.1em} \underline{ba^2\ol{x})\cdot b \ol{a}^2x\ol{b}}\ol{x}y x\Delta_{a,b}\ol{a}^4 \\
=& xy\ol{a}^2 \cdot {}_{b a\ol{x}y}(\ol{b})\cdot {}_{ba^2\ol{x}}(b)\cdot \ol{x}y x\Delta_{a,b}\ol{a}^4,
\end{align*}
}
where we denote the product $(ab)^3$ by $\Delta_{a,b}$.
We can easily check that the curves ${}_{b a\ol{x}y}(b)$ and ${}_{ba^2\ol{x}}(b)$ are as shown in Figure~\ref{F:scc_neigh_abc_complicated}. 
\begin{figure}[htbp]
\subfigure[The curve ${}_{b a\ol{x}y}(b)$.]{\includegraphics[width=37mm]{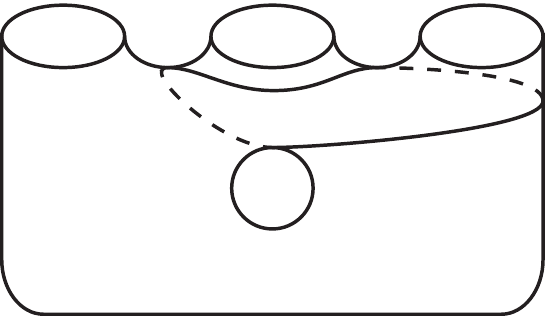}
\label{F:scc_neigh_abc2}}
\subfigure[The curve ${}_{ba^2\ol{x}}(b)$.]{\includegraphics[width=37mm]{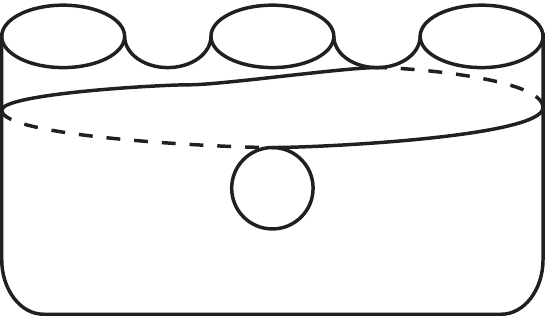}
\label{F:scc_neigh_abc3}}
\caption{Simple closed curves in a neighborhood of $a\cup b\cup c$.}
\label{F:scc_neigh_abc_complicated}
\end{figure}
These curves are disjoint from the curve $a$, in particular the Dehn twists along them are contained in $\Mod(\Sigma)(a)$. 
The map $\Phi_a$ sends the Dehn twists along ${}_{b a\ol{x}y}(b)$ and ${}_{ba^2\ol{x}}(b)$ to $t_{\delta_2}$ and $t_{\delta_3}$, respectively. 
We can thus calculate the monodromy $\Phi_a({}_c(a){}_b(c){}_a(b))$ as follows: 
{\allowdisplaybreaks
\begin{align*}
& \Phi_a({}_c(a){}_b(c){}_a(b)) = \Phi_a(xy\ol{a}^2 \cdot {}_{b a\ol{x}y}(\ol{b})\cdot {}_{ba^2\ol{x}}(b)\cdot \ol{x}y x\Delta_{a,b}\ol{a}^4) \\
=& \delta_3\delta_1 \ol{\delta_2}\delta_3\ol{\delta_3}\delta_1 \delta_3 = \delta_1^2\ol{\delta_2}\delta_3^2.
\end{align*}
}
This completes the proof.
\end{proof}

Meier \cite{Meier} constructed a $(3g,g)$--trisection of the spun manifold $\mathcal{S}(M)$ of a $3$--manifold $M$ with a genus--$g$ Heegaard splitting. 
He also gave an algorithm to obtain a trisection diagram of it from a Heegaard diagram of $M$. 
Using Theorem~\ref{T:descriptionSTviaMCG}, together with this algorithm, we can prove: 

\begin{theorem}\label{T:Meier's ex are simplified}

Meier's trisection on $\mathcal{S}(M)$ is diffeomorphic to one associated with a simplified trisection. 

\end{theorem}

\begin{proof}
Let $(\Sigma_g;\delta,\varepsilon)$ be a genus--$g$ Heegaard diagram of $M$ such that $\varepsilon$--curves are in the standard position (see Figure~\ref{F:Heegaard diagram standard}).
\begin{figure}[htbp]
\includegraphics[height=22mm]{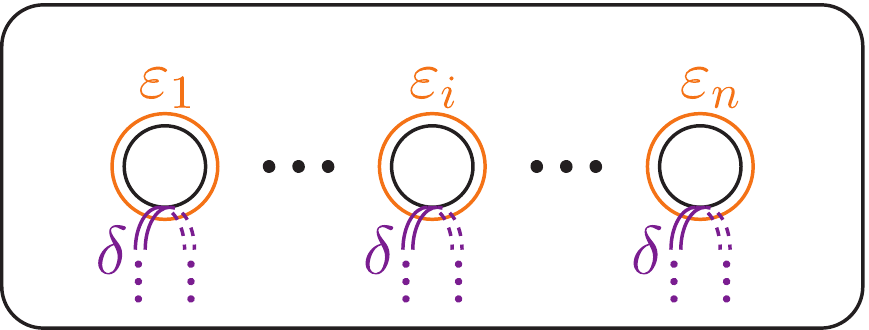}
\caption{A standard Heegaard diagram.}
\label{F:Heegaard diagram standard}
\end{figure}
According to \cite[Theorem 1.4]{Meier}, a trisection diagram of the spun manifold $\mathcal{S}(M)$ can be obtained by replacing a neighborhood of each $\varepsilon_i$ in the Heegaard diagram $(\Sigma_g;\delta,\varepsilon)$ (which is shown in Figure~\ref{F:modificaion HD1}) with a genus--$3$ surface and curves given in Figures~\ref{F:modificaion HD2}, \ref{F:modificaion HD3} and \ref{F:modificaion HD4}, where $\alpha_\delta = \{\alpha_{\delta_1}, \ldots,\alpha_{\delta_g}\}$ and $\alpha_{\delta_j}$ coincides with $\delta_j$ outside of the union of neighborhoods of $\varepsilon_i$'s. 
We will show that the curves $\alpha_1,\ldots,\alpha_{2g},\alpha_{\delta_1},\ldots,\alpha_{\delta_g},\beta_1,\ldots,\beta_{2g},\gamma_1,\ldots,\gamma_{2g}$ satisfy the conditions in Theorem~\ref{T:descriptionSTviaMCG}. 
The statement then follows immediately from Proposition~\ref{T:relation vc diagram}. 
\begin{figure}[htbp]
\subfigure[]{\includegraphics[height=26mm]{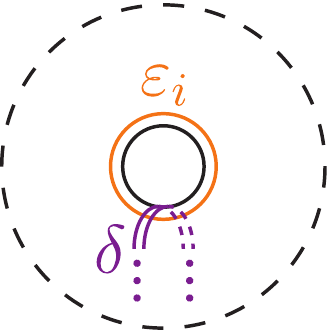}
\label{F:modificaion HD1}}
\subfigure[]{\includegraphics[height=26mm]{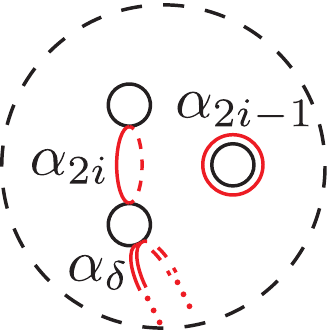}
\label{F:modificaion HD2}}
\subfigure[]{\includegraphics[height=26mm]{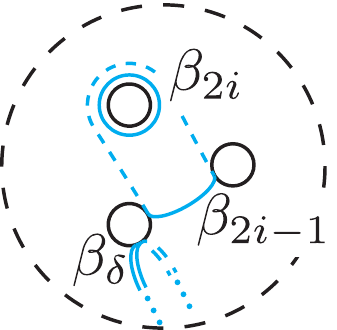}
\label{F:modificaion HD3}}
\subfigure[]{\includegraphics[height=26mm]{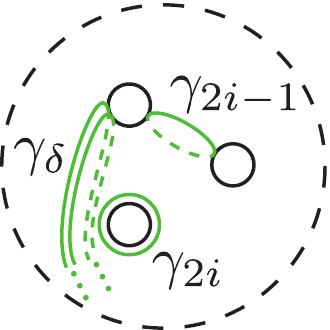}
\label{F:modificaion HD4}}
\caption{Curves in a trisection diagram of $\mathcal{S}(M)$. }
\label{F:modification HD}
\end{figure}
The first condition is obvious, while we can easily deduce the second and the third conditions from the following claim:

\begin{claim}

The element $\mu_i\in \Mod(\Sigma_{\alpha_1,\ldots,\alpha_{i}})$ determined from $\alpha_1,\ldots,\alpha_{2g},\beta_1,\ldots,\beta_{2g},\gamma_1,\ldots,\gamma_{2g}$ is equal to the Dehn twist along ${t_{\beta_{i+1}}(\alpha_{i+1})}$ if $i$ is odd, and equal to the identity if $i$ is even. 

\end{claim}

\noindent
The claim is obvious for $i=0$. 
Suppose that the claim holds for $i=2k-2$. 
Each two of the curves $\alpha_{2k-1},\beta_{2k-1},\gamma_{2k-1}$ intersect on one point. 
One of the boundary component of a regular neighborhood of $\alpha_{2k-1}\cup\beta_{2k-1}\cup\gamma_{2k-1}$ bounds a disk, while the other components are isotopic to ${t_{\beta_{2k}}(\alpha_{2k})}$. 
Thus, we can deduce from Lemma~\ref{T:monodromyalongtriangle} that $\mu_{2k-1}$ is equal to the Dehn twist along ${t_{\beta_{2k}}(\alpha_{2k})}$ (note that we can take orientations of $\alpha_{2k-1},\beta_{2k-1},\gamma_{2k-1}$ so that the algebraic intersections $\alpha_{2k-1}\cdot \beta_{2k-1}$, $\beta_{2k-1}\cdot\gamma_{2k-1}$ and $\gamma_{2k-1}\cdot\alpha_{2k-1}$ are all equal to $1$).

For simplicity, we denote $\alpha_{2k},\beta_{2k},\gamma_{2k}$ by $\alpha,\beta, \gamma$, respectively. 
It is easy to see that the curve $\gamma' = \mu_{2k-1}^{-1}(\gamma)$ intersects with $\beta$ on one point. 
Thus, we can calculate the monodromy $\mu_{2k}$ as follows (in the following calculation, the only underlined part in each line is changed when proceeding to the next line): 
{\allowdisplaybreaks
\begin{align*}
\mu_{2k}=& \Phi_{\alpha}\left(\underline{{}_{\gamma}(\alpha)\cdot \mu_{2k-1}\cdot {}_{\beta}(\gamma')\cdot {}_{\alpha}(\beta)}\right) \\
=& \Phi_{\alpha}\left(\underline{\gamma\alpha\overline{\gamma}}\cdot \beta\alpha\underline{\overline{\beta}\cdot \beta}\hspace{.1em}\underline{\gamma'}\hspace{.1em}\underline{\overline{\beta}\cdot \alpha\beta}\overline{\alpha}\right) \\
=& \Phi_{\alpha}\left(\underline{(\overline{\alpha}}\gamma\alpha)\cdot \beta\alpha\cdot (\underline{\overline{\mu_{2k-1}}}\gamma\underline{\mu_{2k-1}})\cdot (\alpha\beta\underline{\overline{\alpha}^2)}\right) \\
=& \Phi_{\alpha}\left(\gamma\alpha\underline{\beta\alpha\cdot (\beta\overline{\alpha}\overline{\beta})}\cdot \gamma\cdot (\beta\alpha\underline{\overline{\beta})\cdot \alpha\beta}\right) & \mbox{(Note that $\alpha\in \Ker(\Phi_{\alpha})$.)} \\
=& \Phi_{\alpha}\left(\underline{\gamma\alpha\cdot \alpha \cdot \gamma\beta\alpha\cdot (\alpha\beta\overline{\alpha})}\right) \\
=& \Phi_{\alpha}\left(\overline{\alpha}^2(\alpha\gamma)^3 (\alpha\beta)^3\overline{\alpha}^3\right) = 1. 
\end{align*}
}
This completes the proof of the claim. 
\end{proof}

\begin{remark}

In Section~\ref{S:trisection from BLF} we will also obtain a diagram of a simplified $(3,1)$--trisection of $L_p \cong \mathcal{S}(L(p,q))$ constructed from a genus--$1$ simplified broken Lefschetz fibration on it (see Example~\ref{Ex:simplified trisection genus-1SBLF}).
Although some of such simplified trisections (e.g.~those of $S^4$ and $S^1\times S^3\sharp S^2\times S^2$) are diffeomorphic to those constructed by Meier \cite{Meier}, the author does not know whether \emph{any} simplified trisection obtained from a genus--$1$ simplified broken Lefschetz fibration is diffeomorphic to that obtained from a genus--$1$ Heegaard splitting of some lens space. 

\end{remark}

Another application of Theorem~\ref{T:descriptionSTviaMCG} is the classification of $4$--manifolds admitting genus--$2$ simplified trisections:

\begin{theorem}\label{T:classificationgenus2ST}

A $4$--manifold $X$ admits a genus--$2$ simplified trisection if and only if $X$ is diffeomorphic to either $S^2\times S^2$ or a connected sum of $\CP$, $\CPb$ and $S^1\times S^3$ with two summands. 

\end{theorem}

\noindent
As we noted in the beginning of the section, we can deduce this theorem as merely a corollary of the classification of genus--$2$ general trisections in \cite{MZgenus2}. 
Although we only deal with simplified trisections, our proof below relies on easy linear-algebraic calculations, while that in \cite{MZgenus2} involves subtle arguments on configurations of curves in genus--$2$ surfaces.

\begin{proof}[Proof of Theorem~\ref{T:classificationgenus2ST}]
Since $S^1\times S^3, \CP$ and $\CPb$ admit genus--$1$ simplified trisections, any connected sum of them with two summands admits a genus--$2$ simplified trisection. 
Furthermore, we can deduce from \cite[Theorem 1.4]{BaykurSaeki} that $S^2\times S^2$ also admits a genus--$2$ simplified trisection. 
These observations prove the if part of the theorem, so we will prove the only if part. 

Let $f:X\to \R^2$ be a simplified $(2,k)$--trisection ($k=0,1,2$). 
The manifold $X$ is diffeomorphic to $\sharp^2(S^1\times S^3)$ if $k$ is equal to $2$ (see \cite[Remark 5]{GKtrisection}).
Assume that $k$ is not equal to $2$. 
We take vanishing cycles $a_1,a_2$ and $b_j,c_j$ ($j=1,\ldots,g-k$) in $\Sigma = \Sigma_2$ as we took in Theorem~\ref{T:descriptionSTviaMCG}. 
By changing the orientation of $X$ if necessary, we can assume that there are orientations of $a_1, b_1$ and $c_1$ such that the algebraic intersections $a_1\cdot b_1, b_1\cdot c_1$ and $c_1\cdot a_1$ are equal to $1$. 
Let $\Sigma'\subset \Sigma$ be a regular neighborhood of $a_1\cup b_1\cup c_1$, which is a genus--$1$ surface with three boundary components by Lemma~\ref{T:monodromyalongtriangle}. 
We denote the three boundary components of $\Sigma'$ by $\delta_1,\delta_2$ and $\delta_3$ as shown in Figure~\ref{F:vc_triangle2}. 
If $k=1$, we can deduce from Theorem~\ref{T:descriptionSTviaMCG} and Lemma~\ref{T:monodromyalongtriangle} that the product $t_{\delta_1}^2t_{\delta_3}^2t_{\delta_2}^{-1}$ (which is regarded as an element in $\Mod(\Sigma_{a_1})$) preserves the isotopy class of $a_2$. 
Since $\Sigma_{a_1}$ is a torus, it is easy to see that $a_2$ is disjoint from the curves $\delta_1,\delta_2,\delta_3$, and thus from $a_1,b_1,c_1$. 
We can therefore change $f$ by homotopy so that the innermost triangle in $f(\Crit(f))$ is moved to the outermost region (i.e.~the region bounded by the definite fold image). 
We can then apply unwrinkle to the resulting map, which yields a Lefschetz singularity with a trivial vanishing cycle. 
Thus, the manifold $X$ is a blow-up of a manifold admitting $(1,1)$--trisection, which is diffeomorphic to $S^1\times S^3$. 

In what follows, we assume that $k=0$.
Since the genus of $\Sigma$ is $2$, we can obtain $\Sigma$ by capping $\Pa \Sigma'$ by either (1) a genus--$1$ surface with one boundary component and two disks, or (2) an annulus and a disk. 
For the case (1), either of the components $\delta_1$ and $\delta_3$ bounds a disk in $\Sigma$. 
In particular we can apply unwrinkle to $f$ so that the inner indefinite fold image of $f$ becomes a Lefschetz critical value with a trivial vanishing cycle. 
Thus $X$ is a blow-up of a $4$--manifold admitting a $(1,0)$--trisection, which is either $\CP$ or $\CPb$. 

Assume that the components $\delta_1,\delta_2, \delta_3$ bound an annulus and a disk in $\Sigma$. 
By Lemma~\ref{T:monodromyalongtriangle} the monodromy $\psi\in \Mod(\Sigma_{a_1})$ along the loop going between the inner and the outer fold images of $f$ is a (single or fourth power of) Dehn twist along an essential simple closed curve $d$. 
We take an identification of $H_1(\Sigma_{a_1};\Z)$ with $\Z^2$ so that $d$ represents the element $\smat{1 \\ 0}$. 
Let $\smat{p \\ q}, \smat{r\\s}$ and $\smat{t \\u}\in \Z^2$ be elements represented by $a_2,b_2$ and $c_2$, respectively. 
Since $a_2$ intersects each of the curves $b_2$ and $c_2$ on one point, the following equality holds:
\begin{equation}\label{Eq:condition_intersection}
ps-qr= pu-qt = 1 \Longleftrightarrow \begin{pmatrix}
s & -r \\
u & -t
\end{pmatrix}\begin{pmatrix}
p \\ q
\end{pmatrix} = \begin{pmatrix}
1\\1
\end{pmatrix}. 
\end{equation}
If one of the integers $q,s,u$ is equal to zero, one of the curves $a_2,b_2$ and $c_2$ is disjoint from the three curves $a_1, b_1$ and $c_1$. 
Thus we can apply a homotopy to $f$ so that the nested indefinite fold images of $f$ becomes two circles bounding disjoint disks as shown in Figure~\ref{F:homotopy_disjoint}.
\begin{figure}[htbp]
\includegraphics[width=100mm]{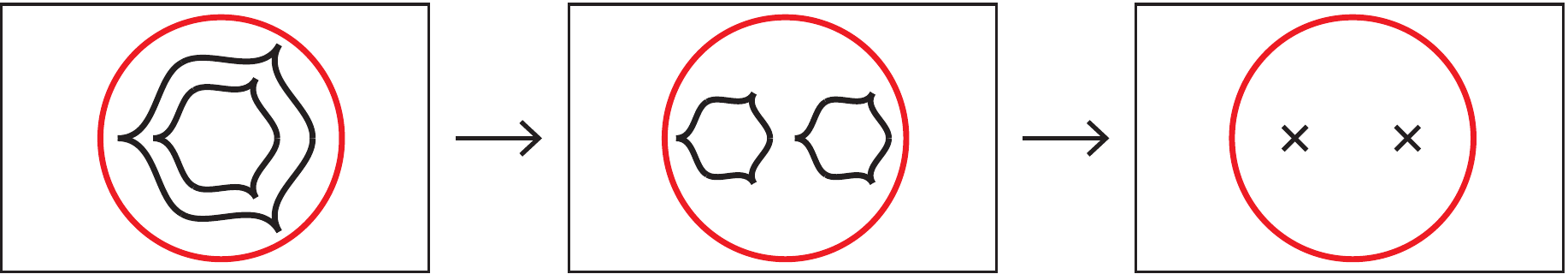}
\caption{Homotopies applied to a genus--$2$ simplified trisection.}
\label{F:homotopy_disjoint}
\end{figure}
Since a regular fiber inside each of the two circles is a genus--$1$ surface, we can further apply wrinkles so that the two indefinite circles become two (possibly achiral) Lefschetz singularities. 
We can thus conclude that the total space of the original trisection is a connected sum of $\CP$ and $\CPb$ with two summands. 
In what follows we will assume $q,s,u \neq 0$. 

If $\psi$ is equal to $t_d$ (i.e.~$\delta_2$ is a boundary component of an annulus in $\overline{\Sigma\setminus \Sigma'}$), we can apply unwrinkle to $f$ so that the inner indefinite fold circle becomes a Lefschetz singularity with a vanishing cycle $d$ (see the first two figures in Figure~\ref{F:homotopy_unwrinkle}). 
\begin{figure}[htbp]
\includegraphics[width=100mm]{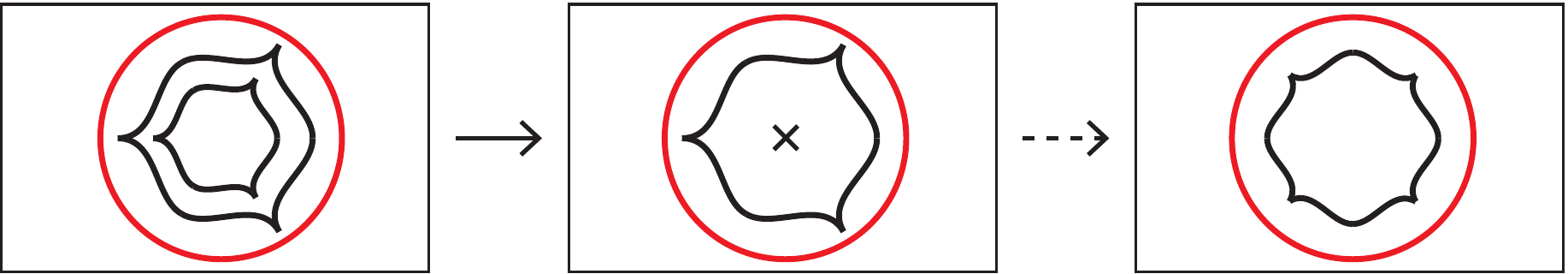}
\caption{Homotopies applied to a genus--$2$ simplified trisection.
Note that the second one can be applied only if one of the integers $q,s$ and $u$ is equal to $\pm 1$.}
\label{F:homotopy_unwrinkle}
\end{figure}
Since the curve ${}_{\overline{d}}(c_2) = \smat{t-u\\u}$ intersects $b_2=\smat{r\\s}$ on one point by Theorem~\ref{T:descriptionSTviaMCG}, we obtain $ur-(t-u)s=\pm 1$ and $ur-ts = -us\pm1$. 
Suppose that $ts-ur = -\det \smat{s&-r\\u&-t}$ is not equal to $0$ and $|q|$ is greater than $1$. 
We can deduce the following inequality from the equality \eqref{Eq:condition_intersection}: 
{\allowdisplaybreaks
\begin{align*}
&2 \leq |q| = \frac{|s-u|}{|-us\pm 1|}\\
\Rightarrow& 2(|us|-1) - (|u|+|s|) \leq 0 \\
\Rightarrow& \left(|s|-\frac{1}{2}\right)\left(|u|-\frac{1}{2}\right) \leq \frac{5}{4}. 
\end{align*}
}
The last inequality implies that either $|s|$ or $|u|$ is equal to $1$ (note that we assume $s,u\neq 0$).  
We can thus conclude that one of the integers $q,s, u$ is equal to $\pm 1$ if $\det \smat{s&-r\\u&-t}$ is not equal to $0$. 
If $\det \smat{s&-r\\u&-t}$ is equal to $0$, we can deduce from the equality \eqref{Eq:condition_intersection} that $\smat{s\\u}$ is equal to $\pm \smat{ 1 \\ 1}$. 

In any case, one of the curves $a_2, b_2$ and $c_2$ intersects $d$ on one point. 
We can thus apply unsink as shown in Figure~\ref{F:homotopy_unwrinkle}. 
Let $d_1,d_2,d_3$ and $d_4$ be vanishing cycles of indefinite fold arcs between two of the four cusps. 
Since two consecutive cycles $d_i$ and $d_{i+1}$ (taking indices mod $4$) intersect on one point, it is easy to deduce that either of the followings occurs for some $i \in \Z/4\Z$: 
\begin{itemize}

\item 
the cycles $d_i$ and $d_{i+2}$ are disjoint, 

\item 
the cycles $d_i$ and $d_{i+2}$ intersect on one point, 

\end{itemize}
In the former case, we can apply flip and slip to $f$ so that the resulting map $\tilde{f}$ does not have cusps. 
Let $D\subset \R^2$ be a open disk inside the innermost fold image of $\tilde{f}$. 
It is easy to see that the complement $X\setminus \tilde{f}^{-1}(D)$ admits a trivial bundle over $S^1$ with a fiber $S^2\times I$, and $X$ can be obtained from this bundle by attaching two copies of the trivial bundle $S^2\times D^2$ over $D^2$ by fiber-preserving diffeomorphisms. 
We can thus conclude that $X$ is an $S^2$--bundle over $S^2$. 

Lastly, we will prove that the monodromy $\psi$ never be equal to $t_d^4$ (i.e.~$\delta_2$ never bound a disk in $\overline{\Sigma\setminus \Sigma'}$) under the assumption that $q,s,u\neq 0$. 
To do this, suppose that $\psi$ would be equal to $t_d^4$.
The curve ${}_{\overline{d}^4}(c_2) = \smat{t-4u\\u}$ would intersect $b_2=\smat{r\\s}$ on one point by Theorem~\ref{T:descriptionSTviaMCG} and Lemma~\ref{T:monodromyalongtriangle}. 
Thus $\det \smat{s&-r\\u&-t}=-4us \pm 1$ would not be equal to $0$, and we obtain:
{\allowdisplaybreaks
\begin{align*}
&1 \leq |q| = \frac{|s-u|}{|-4us\pm 1|}\\
\Rightarrow& (4|us|-1) - (|u|+|s|) \leq 0 \\
\Rightarrow& \left(2|s|-\frac{1}{2}\right)\left(2|u|-\frac{1}{2}\right) \leq \frac{5}{4}. 
\end{align*}
}
This contradicts the assumption that $s$ and $u$ are not equal to $0$. 
\end{proof}

\section{Trisection diagrams from broken Lefschetz fibrations}\label{S:trisection from BLF}

In \cite{BaykurSaeki} the authors gave an explicit algorithm to obtain a simplified trisection from a directed broken Lefschetz fibration. 
In this section, we first explain how to obtain a trisection \emph{diagram} of a simplified trisection obtained by this algorithm. 
Although our method works for Lefschetz fibrations and directed broken Lefschetz fibrations, we will only focus on simplified broken Lefschetz fibrations for simplicity. 
We then apply this method to genus--$1$ simplified broken Lefschetz fibrations. 

We begin with a brief review of the algorithm in \cite{BaykurSaeki} mentioned above. 
Let $f:X\to S^2$ be a genus--$g$ simplified broken Lefschetz fibration with $k$ Lefschetz singularities. 
We first take a decomposition $S^2=D_1\cup D_2\cup D_3$, where $D_1$ is a disk including all the critical values of $f$, $D_3\subset S^2\setminus D_1$ is a small disk neighborhood of a regular value with a lower genus fiber and $D_2$ is an annulus between the two disks. 
We take identifications $D^2\cong S^1\times [-1,1]$ and $f^{-1}(D^2)\cong S^1\times [-1,1]\times \Sigma_{g-1}$ under which the restriction $f|_{f^{-1}(D_2)}$ is the projection onto the former components. 
Let $h:\Sigma_{g-1}\to [1,2]$ be a Morse function with $2g-2$ index--$1$ critical points, one index--$0$ critical point and one index--$2$ critical point. 
Using $h$ we define $\varphi:[-1,1]\times \Sigma_{g-1}\to [1,3]$ as $\varphi(t,x) = 1+(1-t^2)h(x)$\footnote{Here we take a different $\varphi$ from that in \cite{BaykurSaeki} in order to make it easy to take a vector field giving vanishing cycles.} and $f_1:X\to \R^2$ as follows: $f_1$ is a composition of $f|_{f^{-1}(D_i)}$ and a suitable diffeomorphism from $D_i$ to a unit disk $D^2\subset \R^2$ on $f^{-1}(D_i)$ for $i=1,3$, and is a composition of $\id_{S^1}\times \varphi$ and a suitable diffeomorphism from $S^1\times [1,3]$ to an annulus in $\R^2$ on $f^{-1}(D_2)\cong S^1\times [-1,1]\times \Sigma_{g-1}$. 
The critical value set of $f_1$ is as shown in Figure~\ref{F:homotopy_algorithmBS1}, where it has $2g-1$ outward-directed indefinite fold circles. 
We first apply $\mathrm{R2}$--moves twice to interchange the first and the second innermost circles, yielding the critical value set shown in Figure~\ref{F:homotopy_algorithmBS2}. 
The Lefschetz singularities can be moved to the outside of the innermost region by a homotopy, and the resulting critical value set is shown in Figure~\ref{F:homotopy_algorithmBS3}.
We can then apply flip and slip and unsink to obtain an indefinite fold circle with three cusps as shown in Figure~\ref{F:homotopy_algorithmBS5}. 
Finally, we can obtain a simplified trisection by applying wrinkles and pushing Lefschetz singularities inside successively. 
The resulting trisection has $2g-1$ indefinite fold circles without cusps and $k+2$ indefinite fold circles with three cusps, so it is a simplified $(2g+k+1,2g-1)$--trisection. 
\begin{figure}[htbp]
\subfigure[]{\includegraphics[width=40mm]{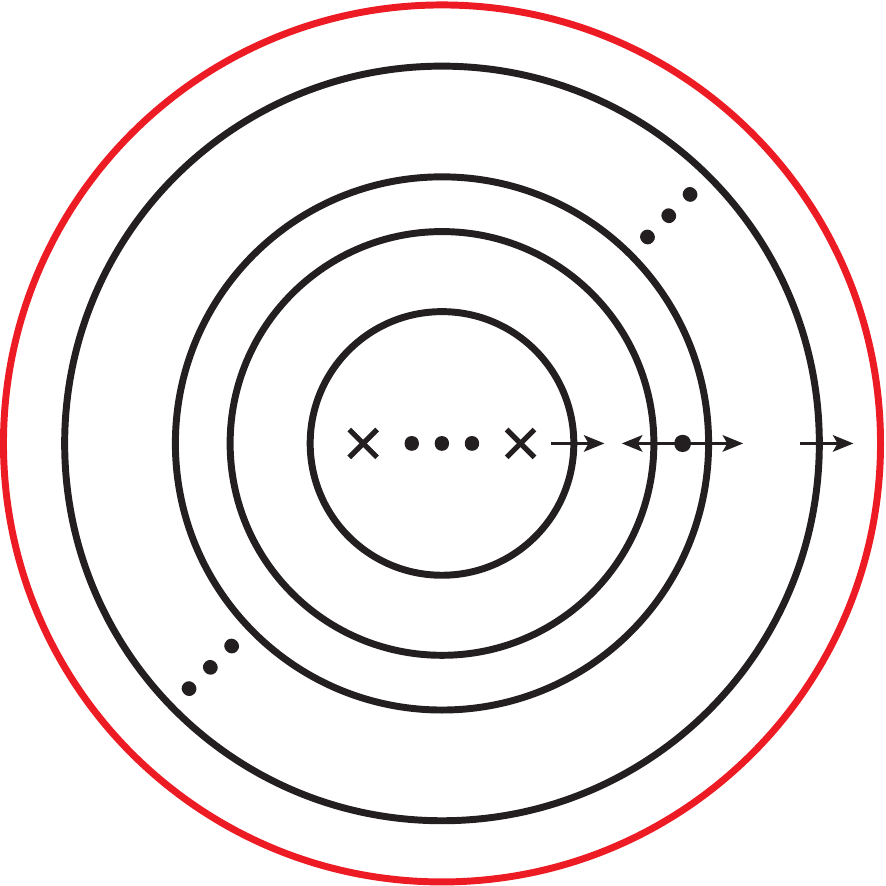}
\label{F:homotopy_algorithmBS1}}
\subfigure[]{\includegraphics[width=40mm]{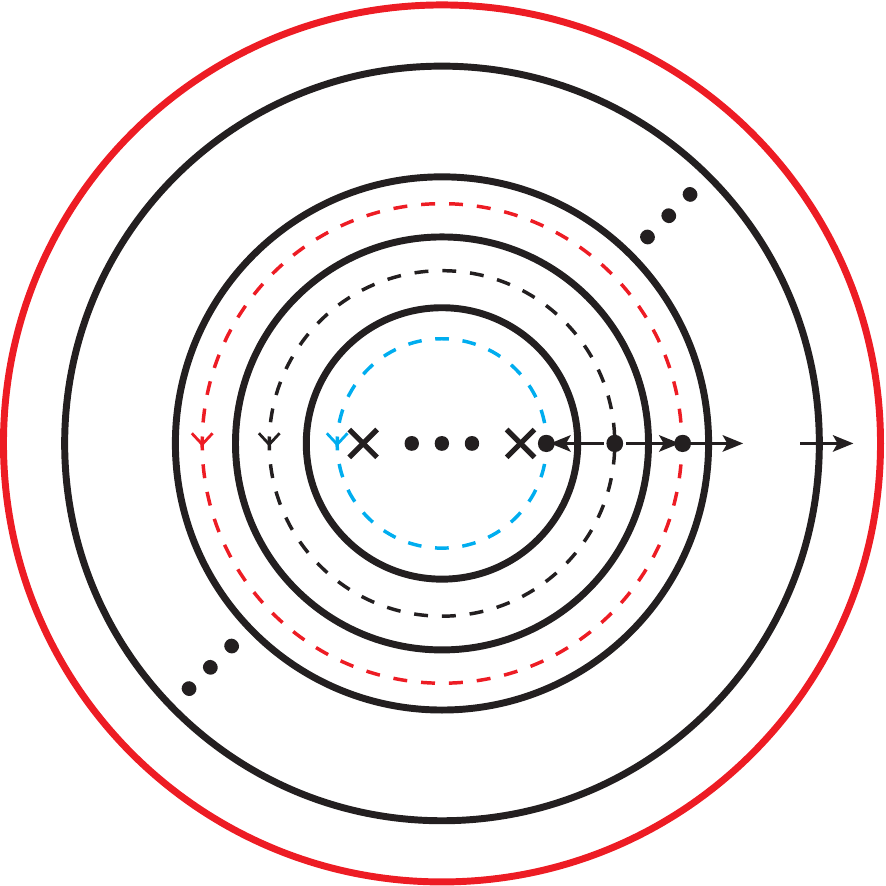}
\label{F:homotopy_algorithmBS2}}
\subfigure[]{\includegraphics[width=40mm]{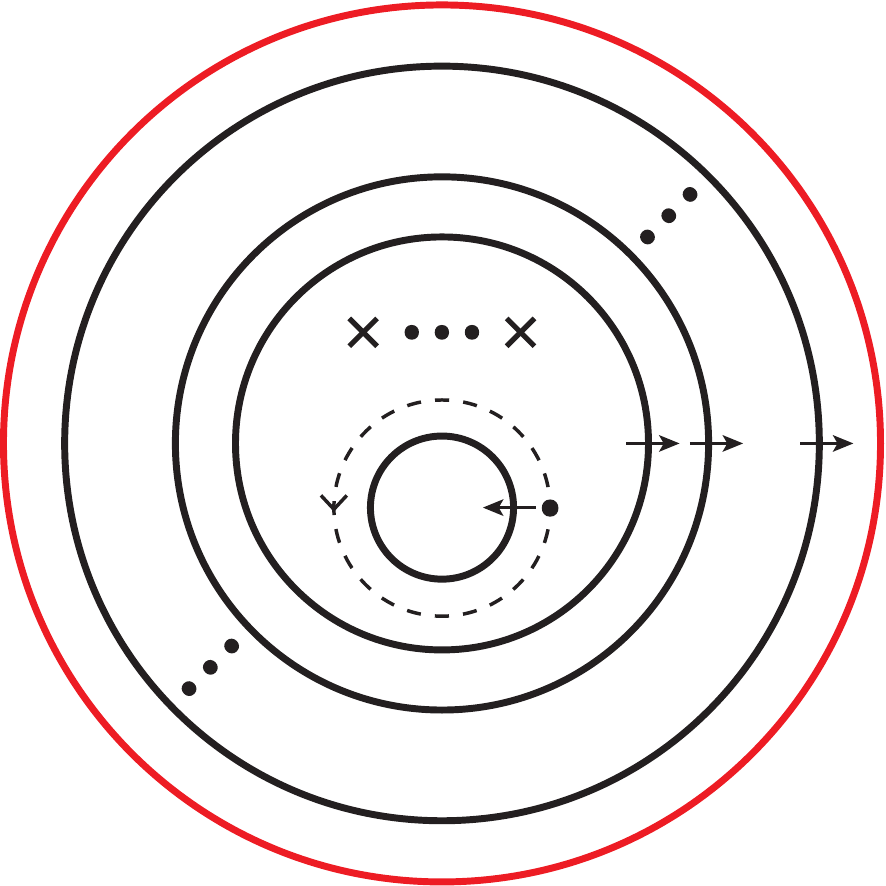}
\label{F:homotopy_algorithmBS3}}
\subfigure[]{\includegraphics[width=40mm]{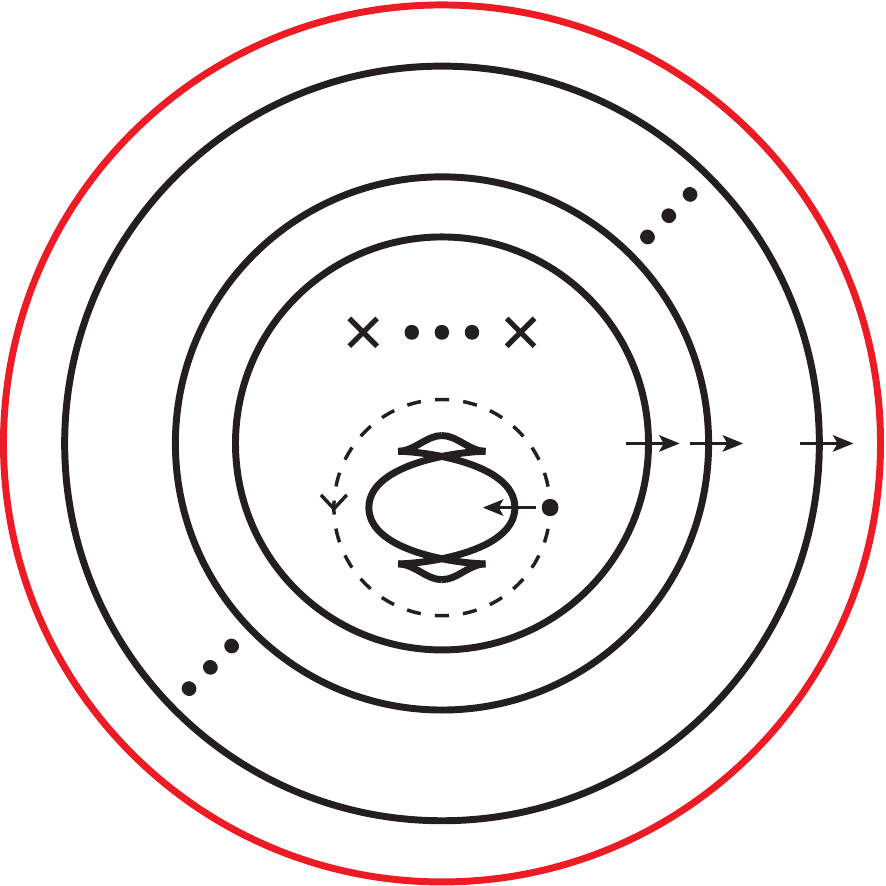}
\label{F:homotopy_algorithmBS4}}
\subfigure[]{\includegraphics[width=40mm]{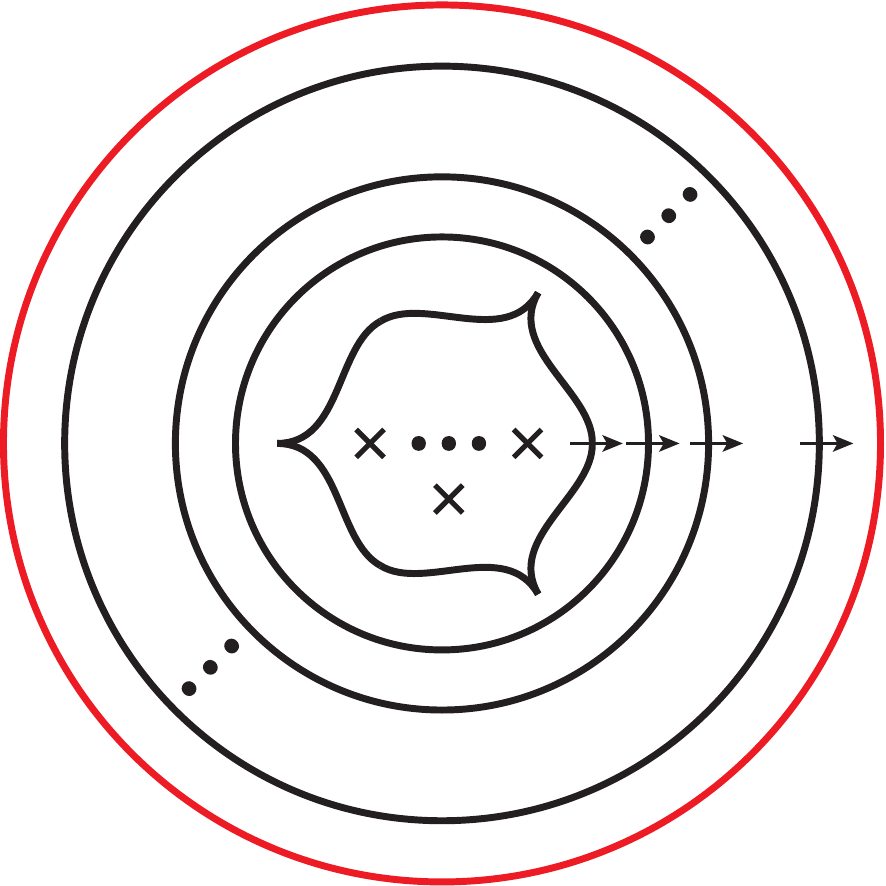}
\label{F:homotopy_algorithmBS5}}
\subfigure[]{\includegraphics[width=40mm]{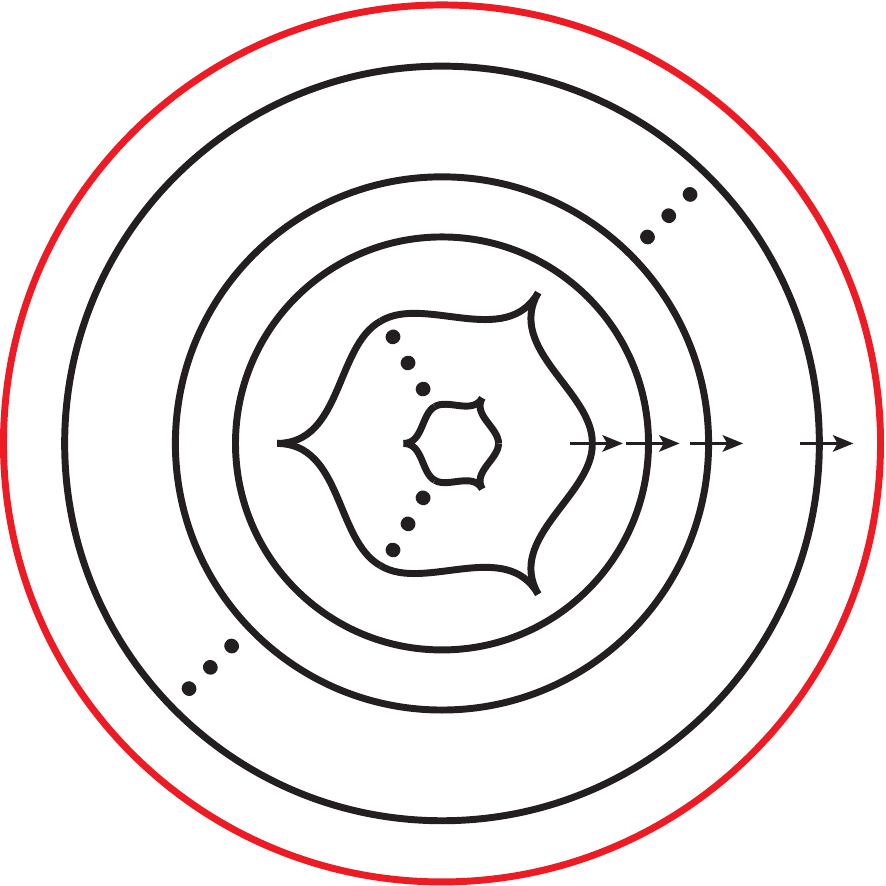}
\label{F:homotopy_algorithmBS6}}
\caption{Critical value sets appearing in the algorithm to obtain a trisection.}
\label{F:homotopy_algorithmBS}
\end{figure}

We denote the map appearing in the algorithm with critical value set Figures~\ref{F:homotopy_algorithmBS2},\ldots,\ref{F:homotopy_algorithmBS6} by $f_{2},\ldots,f_{6}$, respectively. 
In order to obtain a trisection \emph{diagram} associated with a simplified trisection constructed above, we have to get vanishing cycles of $f_1$ and know how these are changed in each of the homotopies applied in the algorithm. 
Let $G'$ be a Riemannian metric on $\Sigma_{g-1}$ so that the pair $(h,G')$ satisfies the Morse--Smale condition (for details of this condition, see \cite{BanyagaHurtubise}) and $G=(h(x)dt^2)\oplus G'$, which is a Riemannian metric on $[-1,1]\times \Sigma_{g-1}$.
We first determine stable and unstable manifolds of $\varphi$ with respect to the metric $G$. 

\begin{lemma}\label{T:stablemfd_phi}

Let $\varepsilon>0$ be a sufficiently small positive number. 

\begin{enumerate}

\item 
The fiber $\varphi^{-1}(2+\varepsilon)$ is diffeomorphic to a closed surface obtained by attaching two copies of $h^{-1}([1+\varepsilon,2])$ by the identity along the boundary. 

\item 
For the index--$0$ critical point $x_0\in \Crit(h)$, the intersection between $\varphi^{-1}(2+\varepsilon)$ and the stable manifold $W^s(0,x_0)$ of $(0,x_0)$ is the boundary of the two copies of $h^{-1}([1+\varepsilon,2])$ under the identification given in the proof of (1). 

\item 
For an index--$1$ critical point $x\in \Crit(h)$, the intersection between $\varphi^{-1}(2+\varepsilon)$ and the unstable manifold $W^u(0,x)$ of $(0,x)$ is the union of two copies of the intersection between $h^{-1}([1+\varepsilon,2])$ and the unstable manifold $W^u(x)$ of $x$ under the identification given in the proof of (1). 

\end{enumerate}

\end{lemma}

\begin{proof}
For simplicity, we denote the two copies of the surface $h^{-1}([1+\varepsilon,2])$ by $\Sigma_1'$ and $\Sigma_2'$. 
It is easy to see that the map $\Phi:\varphi^{-1}(2+\varepsilon) \to \Sigma_1'\cup_{\id}\Sigma_2'$ defined as
\[
\Phi(t,x) = \begin{cases}
x \in \Sigma_1' & (t\geq 0) \\
x \in \Sigma_2' & (t\leq 0) 
\end{cases}
\]
is a diffeomorphism, in particular the statement (1) holds. 
%In what follows we identify the two surfaces $\varphi^{-1}(2+\varepsilon)$ and $\Sigma_1'\cup_{\id}\Sigma_2'$ via $\Phi$. 
Let $c_p(s)$ be the integral curve of $\grad(h)$ with the initial point $p\in \Sigma_{g-1}$ and $C_{(t,q)}(s) = \left(C^1_{(t,q)}(s),C^2_{(t,q)}(s)\right)$ be the integral curve of $\grad(\varphi)$ with the initial point $(t,q)\in \varphi^{-1}(2+\varepsilon)$.
Since the gradient $\grad(\varphi)$ of $\varphi$ is equal to $-2t\frac{\Pa}{\Pa t} + (1-t^2)\grad(h)$, the components $C^1_{(t,q)}(s)$  and $C^2_{(t,q)}(s)$ are respectively equal to $t\exp\left(-2s\right)$ and $\D c_q\left(s-\frac{t^2}{4}\left(1-\exp(-4s)\right)\right)$. 
Thus, the point $(t,q) \in h^{-1}(1+\varepsilon)$ is contained in $W^s(0,x_0)$ if and only if $\D \lim_{s\to -\infty} c_q\left(s-\frac{t^2}{4}\left(1-\exp(-4s)\right)\right) = x_0$ and $\D \lim_{s\to -\infty}t\exp(-2s)=0$. 
We can deduce from the second equality that $t$ is equal to $0$. 
Since $\Phi(0,q)$ is contained in $W^s(x_0)$ (in particular $\D \lim_{s\to -\infty} c_q\left(s\right) = x_0$), the statement (2) holds. 
Similarly, for an index--$1$ critical point $x\in \Crit(h)$, the point $(t,q) \in h^{-1}(1+\varepsilon)$ is contained in $W^u(0,x)$ if and only if $\D \lim_{s\to \infty} c_q\left(s-\frac{t^2}{4}\left(1-\exp(-4s)\right)\right) = x$ and $\D \lim_{s\to \infty}t\exp(-2s)=0$. 
The second equality holds for any $t$ and the first equality implies that $q$ is contained in $W^u(x)$. 
This completes the proof of the statement (3). 
\end{proof}

By Lemma~\ref{T:stablemfd_phi} we can obtain vanishing cycles of the map $f_1$ in a regular fiber on the second innermost annulus (i.e.~the fiber on the dot in Figure~\ref{F:homotopy_algorithmBS1}):  
the blue curve in Figure~\ref{F:vcMorsevarphi}, which we denote by $c$, is a vanishing cycle of the second innermost fold circle, red curves are vanishing cycles of the outer fold circles, and the pair of the two shaded disks is a neighborhood of a vanishing set of the innermost fold circle. 
\begin{figure}[htbp]
\includegraphics[width=100mm]{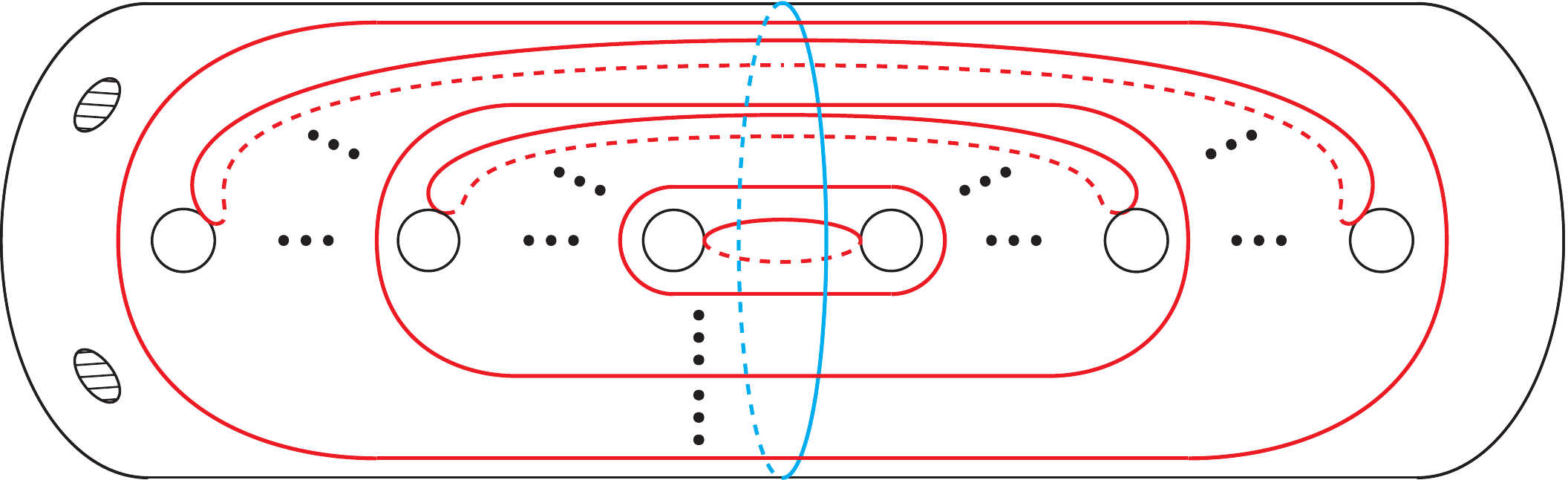}
\caption{Vanishing cycles of $f_1$.}
\label{F:vcMorsevarphi}
\end{figure}
A regular fiber of $f_2$ on the innermost annulus (i.e.~the fiber on the dot in Figure~\ref{F:homotopy_algorithmBS2}) can be obtained by applying surgery along the two shaded disks in Figure~\ref{F:vcMorsevarphi}. 
We denote this regular fiber by $\Sigma$. 
A vanishing cycle of the second innermost fold circle, which we denote by $d$, is parallel to the boundary of one of the shaded disks in Figure~\ref{F:vcMorsevarphi}. 
while the other vanishing cycles are the same as those of $f_1$. 

In order to get vanishing cycles of the maps $f_2,\ldots,f_6$, we need to determine a monodromy $\varphi_2\in \Mod(\Sigma)$ along a dotted black circle in Figure~\ref{F:homotopy_algorithmBS2}. 
Let $c_1,\ldots,c_k$ be vanishing cycles of the Lefschetz singularities with respect to some Hurwitz path system. 
We regard these cycles as curves in $\Sigma_c$, which is a surface obtained by applying surgery along $c\subset \Sigma$. 

\begin{proposition}\label{T:variety R2move tildef}

Suppose that $g$ is greater than $2$.
For any element $\varphi\in \Ker(\Phi_d)$ satisfying the condition $\Phi_c(\varphi) = t_{c_k}\circ\cdots \circ t_{c_1} \in \Mod(\Sigma_c)$, there exist $\mathrm{R2}$--moves applied to $f_1$ in the algorithm above such that the resulting monodromy $\varphi_2$ is equal to $\varphi$. 

\end{proposition}

\begin{proof}
We can prove the proposition in a way similar to that of the proof of \cite[Theorem 3.9]{HayanoR2} provided that the following claim holds:

\begin{claim}

The subgroup $\Ker(\Phi_c)\cap \Ker(\Phi_d) \subset \Mod(\Sigma)(c,d)$ is generated by the following set: 
\[
\left\{t_{\tilde{\delta}(\eta)}\circ t_c^{-1}\circ t_d^{-1}\in \Mod(\Sigma)(c,d)\left|~\begin{minipage}[c]{45mm}
$\eta\in \Pi\left(\Sigma_{c,d}\setminus \{v_i,w_j\},v_k,w_l\right)$

\vspace{.3em}

$\{i,k\}=\{j,l\}=\{1,2\}$
\end{minipage} \right. \right\}, 
\]
where $v_1, v_2\in \Sigma_{c,d}$ (resp.~$w_1, w_2\in \Sigma_{c,d}$) are the origins of the two disks attached in surgery along $c$ (resp.~$d$), and $\Pi\left(\Sigma_{c,d}\setminus \{v_i,w_j\},v_k,w_l\right)$ and $\tilde{\delta}(\eta)$ are defined in the same way as in \cite[Section 3]{HayanoR2}. 

\end{claim}

\noindent
Let $F_{v_1,v_2}$ and $F_{w_1,w_2}$ be the forgetting map defined on $\Mod(\Sigma_{c,d};v_1,v_2,w_1,w_2)$. 
To prove the claim, we first observe that the restriction 
\[
\Phi_c^\ast\circ \Phi_d^\ast|_{\Ker(\Phi_c)\cap \Ker(\Phi_d)}:\Ker(\Phi_c)\cap \Ker(\Phi_d)\to \Ker(F_{v_1,v_2})\cap \Ker(F_{w_1,w_2})
\]
is an isomorphism (see the proof of \cite[Lemma 3.1]{HayanoR2}). 
We denote the connected component of $\Sigma_{c,d}$ containing $v_1,w_1,w_2$ (resp.~$v_2$) by $\Sigma'$ (resp.~$\Sigma''$). 
It is easy to see that $\Ker(F_{v_1,v_2})\cap \Ker(F_{w_1,w_2})$ is contained in the kernel of $F_{v_1}:\Mod(\Sigma';v_1,w_1,w_2)\to \Mod(\Sigma';w_1,w_2)$, where we regard $\Mod(\Sigma';v_1,w_1,w_2)$ as a subgroup of the group $\Mod(\Sigma_{c,d};v_1,v_2,w_1,w_2)$ in the obvious way. 
Since $\Sigma'$ is not a sphere, in particular a connected component of $\Diff^+(\Sigma',w_1,w_2)$ is contractible, the kernel of $F_{v_1}:\Mod(\Sigma';v_1,w_1,w_2)\to \Mod(\Sigma';w_1,w_2)$ is isomorphic to $\pi_1(\Sigma'\setminus\{w_1,w_2\},v_1)$.  
The intersection $\Ker(F_{v_1})\cap \Ker(F_{w_1,w_2})$ is then isomorphic to the kernel of the map $i_\ast :\pi_1(\Sigma'\setminus\{w_1,w_2\},v_1) \to \pi_1(\Sigma',v_1)$ since the diagram
\[
\begin{CD}
\pi_1(\Sigma'\setminus\{w_1,w_2\},v_1)@>P_1>> \Mod(\Sigma';v_1,w_1,w_2) \\
@Vi_\ast VV @VVF_{w_1,w_2}V \\
\pi_1(\Sigma',v_1)@>P_2>> \Mod(\Sigma';v_1)
\end{CD}
\]
commutes and the pushing map $P_2:\pi_1(\Sigma',v_1)\to  \Mod(\Sigma';v_1)$ is injective (note that it would not hold if the genus of $\Sigma'$ were equal to $1$). 
The rest of the proof of the claim is quite similar to that of \cite[Theorem 3.4]{HayanoR2}, so we leave it to the reader. 
\end{proof}

\noindent
Note that Proposition~\ref{T:variety R2move tildef} would not hold without the assumption on $g$. 
In order to get a monodromy $\varphi_2$ of $f_2$ with small fiber genera, we need to take a section on the annulus bounded by the blue and red dotted circles in Figure~\ref{F:homotopy_algorithmBS2} which intersects with the higher-genus connected component of a regular fiber on the innermost region. 
Using such a section, we can take a lift $\tilde{\varphi}_2\in\Mod(\Sigma;x)$ of the monodromy $\varphi_2$, and we can prove the following in the same way as that of the proof of Proposition~\ref{T:variety R2move tildef}: 

\begin{proposition}\label{T:variety R2move tildef small g}

For any element $\tilde{\varphi}\in \Ker(\Phi_d:\Mod(\Sigma,x)(d)\to \Mod(\Sigma_d,x))$ satisfying the condition $\Phi_c(\tilde{\varphi}) = t_{c_k}\circ\cdots \circ t_{c_1} \in \Mod(\Sigma_c,x)$, there exist $\mathrm{R2}$--moves applied to $f_1$ in the algorithm above such that the resulting lift $\tilde{\varphi}_2$ is equal to $\tilde{\varphi}$. 

\end{proposition}

We can regard the vanishing cycles $c_1,\ldots,c_k$ as curves in $\Sigma$ and these are also vanishing cycles of the Lefschetz singularities of $f_3$. 
As shown in \cite[Figure 6]{HayanoR2}, a regular fiber on the dot in the upper triangle in Figure~\ref{F:referencepaths_slip1}, which we denote by $\tilde{\Sigma}$, can be obtained by applying surgery on a pair of disks in $\Sigma$. 
\begin{figure}[htbp]
\subfigure[]{\includegraphics[width=35mm]{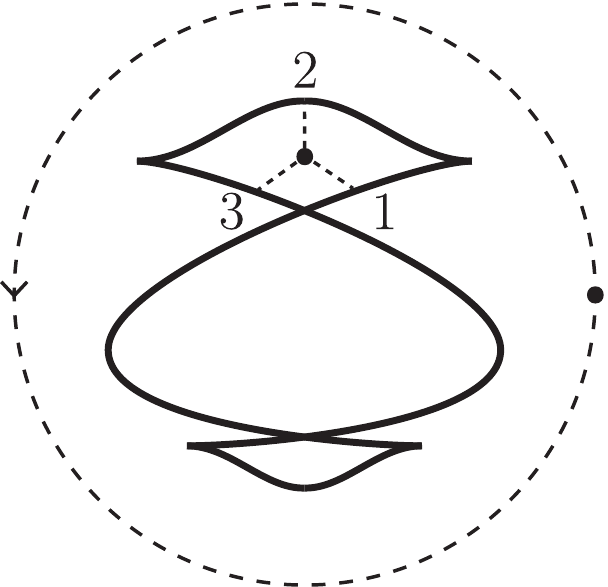}
\label{F:referencepaths_slip1}
}
\subfigure[]{\includegraphics[width=35mm]{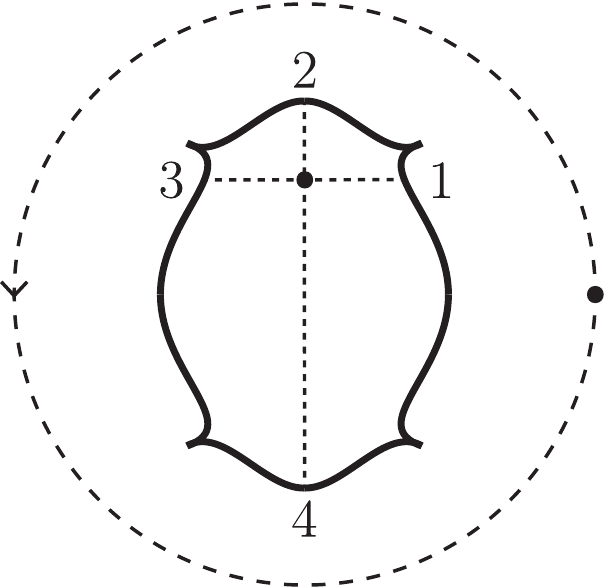}
\label{F:referencepaths_slip2}
}
\caption{Parts of critical value sets \subref{F:referencepaths_slip1} before and \subref{F:referencepaths_slip2} after an $\mathrm{R2}$--move applied to $f_4$. }
\label{F:referencepaths_slip}
\end{figure}
We denote by $\varphi_4\in \Mod(\Sigma)$ the monodromy along the dotted circle in Figure~\ref{F:referencepaths_slip}, which is equal to $\varphi_2\circ t_{c_1}^{-1}\circ \cdots \circ t_{c_k}^{-1}$. 
Let $e_i\subset \tilde{\Sigma}$ be a vanishing cycle associated with a reference path in Figure~\ref{F:referencepaths_slip2} labeled by $i$ ($i=1,2,3,4$). 
Although we can easily obtain $e_1,e_2,e_3$ (see \cite[Figure 6]{HayanoR2}), $e_4$ depends on $\varphi_4$: 

\begin{proposition}\label{T:variety R2move f_d}

Suppose that $g$ is greater than $2$. 
For any element $\psi \in \Ker(\Phi_{e_3})\cap \Mod(\tilde{\Sigma})(e_1)$ satisfying the condition $\Phi_{e_1}(\psi)=\varphi_4$, there exists an $\mathrm{R2}$--move applied to $f_4$ such that the resulting vanishing cycle $e_4$ is equal to $\psi(e_2)$. 

\end{proposition}

\noindent
Proposition \ref{T:variety R2move f_d} immediately follows from \cite[Theorem 4.1]{HayanoR2}. 

Again, in order to get $e_4$ when the genus of a fiber is small, we have to take sections of $f_4$ on the disk bounded by the dotted circle in Figure~\ref{F:referencepaths_slip}. 
The complement $\tilde{\Sigma}\setminus (e_1\cup e_3)$ has two connected components $\Sigma_h$ and $\Sigma_l$, where the genus of $\Sigma_h$ is one larger than that of $\Sigma_l$. 
We take a section which intersects with $\Sigma_l$ if $g=2$, while we take four sections $\sigma_l^1, \sigma_l^2,\sigma_l^3$ and $\sigma_h$ so that $\sigma_h$ (resp.~$\sigma_l^i$) intersects with $\Sigma_h$ (resp.~$\Sigma_l$) if $g=1$. 
Using the sections we can take a lift $\tilde{\varphi}_4$ of the monodromy $\varphi_4$, which is contained in $\Mod(\tilde{\Sigma};x)(e_1,e_3)$ if $g=2$, or in $\Mod(\tilde{\Sigma};x_1,x_2,x_3,x_4)(e_1,e_3)$ if $g=1$.

\begin{proposition}\label{T:variety R2move f_d smallg}

For any element $\tilde{\psi}$, which is an element in $\Ker(\Phi_{e_3})\cap \Mod(\tilde{\Sigma};x)(e_1)$ if $g=2$ or in $\Ker(\Phi_{e_3})\cap \Mod(\tilde{\Sigma};x_1,x_2,x_3,x_4)(e_1)$ if $g=1$, satisfying the condition $\Phi_{e_1}(\tilde{\psi})=\tilde{\varphi}_4$, there exists an $\mathrm{R2}$--move applied to $f_4$ such that the resulting vanishing cycle $e_4$ is equal to $\tilde{\psi}(e_2)$. 

\end{proposition}

\noindent
We can prove Proposition \ref{T:variety R2move f_d smallg} as the author proved the theorems in \cite[Section 5]{HayanoR2}. 

We can regard the vanishing cycles $c_1,\ldots,c_k$ as curves in $\tilde{\Sigma}$, and these are also vanishing cycles of Lefschetz singularities of $f_5$. 
It is also easy to obtain a vanishing cycle of a Lefschetz singularity appearing when applying unsink to obtain $f_5$. 
The simplified trisection $f_6$ can be obtained by applying wrinkles and pushing Lefschetz singularities across indefinite folds. 
Since we can easily understand how vanishing cycles are changed by these moves (for wrinkles, see \cite[Figure 8]{BehrenHayano}), we can eventually obtain a trisection diagram associated with the trisection $f_6$. 

\begin{example}\label{Ex:simplified trisection genus-1SBLF}

Here we will apply the algorithm above to genus--$1$ simplified broken Lefschetz fibrations without Lefschetz singularities. 
Such fibrations were first given in \cite{ADK_BLF} and then completely classified in \cite{BaykurKamada,Hayanogenus1SBLF}: a $4$--manifold $X$ admits a genus--$1$ simplified broken Lefschetz fibration without Lefschetz singularities if and only if $X$ is diffeomorphic to one of the manifolds $S^4, S^1\times S^3\sharp S^2\times S^2, S^1\times S^3\sharp \CP\sharp \CPb, L_n$ and $L_n'$, where $L_n$ and $L_n'$ ($n\geq 2$) are $4$--manifolds introduced in \cite{Pao}\footnote{In this paper, we assume that a simplified broken Lefschetz fibration has indefinite folds.}. 
For simplicity of the notations, we put $L_1=L_1' = S^4$, $L_0=S^1\times S^3\sharp S^2\times S^2$ and $L_0'=S^1\times S^3\sharp \CP\sharp \CPb$. 

Let $f:X\to S^2$ be a genus--$1$ simplified broken Lefschetz fibration without Lefschetz singularities and $f_1:X\to \R^2$ be a map appearing when applying the algorithm to $f$ (the critical value set of $f_1$ is shown in Figure~\ref{F:critv_genus1_1}).
Let $\nu(f(\Crit(f)))$ be a tubular neighborhood of $f(\Crit(f))\subset S^2$. 
Since the composition of the restriction $f|_{\nu(f(\Crit(f)))}$ and the natural projection $\nu(f(\Crit(f)))\to f(\Crit(f))$ is a trivial bundle, the composition of the restriction of $f_1$ on the preimage of the shaded annulus in Figure~\ref{F:critv_genus1_1} and a retraction to one of the boundary components of the annulus is also a trivial bundle. 
Thus we can take a section $\sigma$ of $f_1$ over the shaded annulus in Figure~\ref{F:critv_genus1_1} so that the resulting lifted monodromies (which are contained in the mapping class groups of pointed surfaces) along the two boundary components are both trivial. 
By proposition~\ref{T:variety R2move tildef small g} we can apply $\mathrm{R2}$--moves to $f_1$ so that the monodromy along the boundary of the shaded disk in Figure~\ref{F:critv_genus1_2} is trivial. 
\begin{figure}[htbp]
\subfigure[]{\includegraphics[width=29mm]{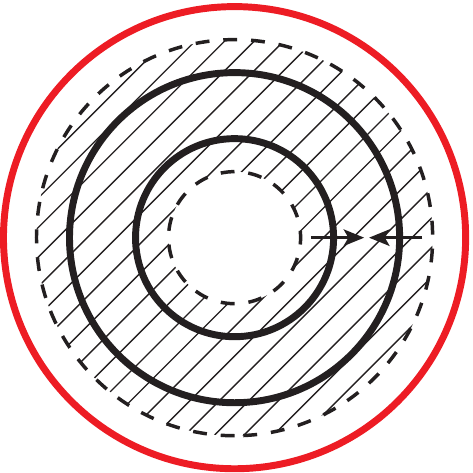}
\label{F:critv_genus1_1}}
\subfigure[]{\includegraphics[width=29mm]{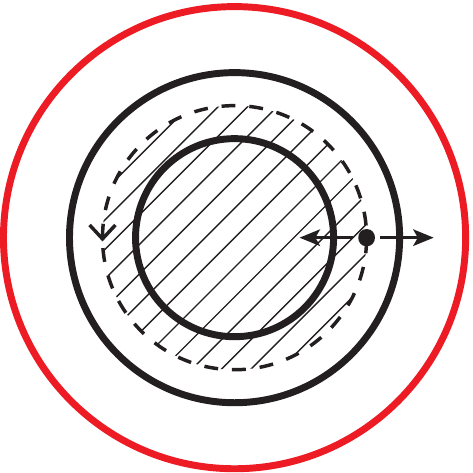}
\label{F:critv_genus1_2}}
\subfigure[]{\includegraphics[width=29mm]{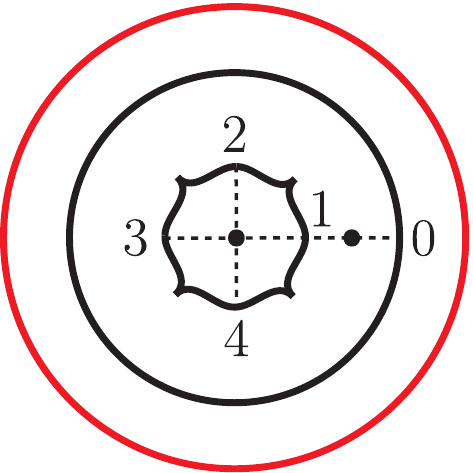}
\label{F:critv_genus1_3}}
\subfigure[]{\includegraphics[width=29mm]{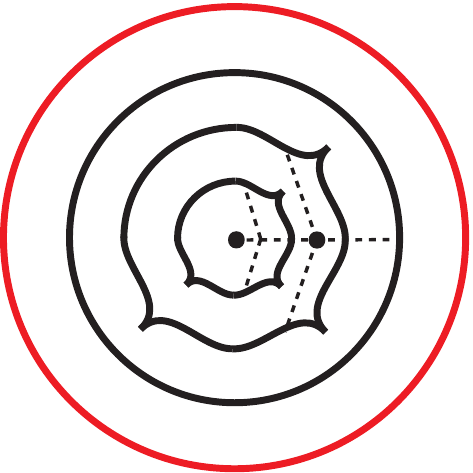}
\label{F:critv_genus1_4}}
\caption{Critical value sets of \subref{F:critv_genus1_1} $f_1$, \subref{F:critv_genus1_2} $f_2$, \subref{F:critv_genus1_3} $f_4$ and \subref{F:critv_genus1_4} $f_6$.}
\label{F:critv_genus1}
\end{figure}

There exists a genus--$1$ simplified broken Lefschetz fibration $f_0:L_0\to S^2$ such that a section $\sigma$ above can be taken so that it can be extended to the inside of the shaded annulus in Figure~\ref{F:critv_genus1_1}.
We take four sections $\sigma_1,\ldots,\sigma_4$ of $f_{0,2}$ over the shaded disk in Figure~\ref{F:critv_genus1_2} so that one of them intersects the torus component of the central fiber, while the others intersect the sphere component of it. 
We also take two regular values of $f_{0,4}$ as shown in Figure~\ref{F:critv_genus1_3}. 
We can obtain vanishing cycles $e_0,\ldots,e_4$ associated with the reference paths in Figure~\ref{F:critv_genus1_3}, together with points corresponding to sections as shown in Figures~\ref{F:loop torus} and \ref{F:loop genus2}. 
\begin{figure}[htbp]
\subfigure[]{\includegraphics[height=24mm]{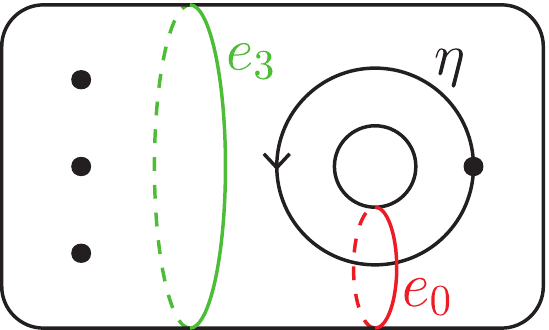}
\label{F:loop torus}}
\subfigure[]{\includegraphics[height=24mm]{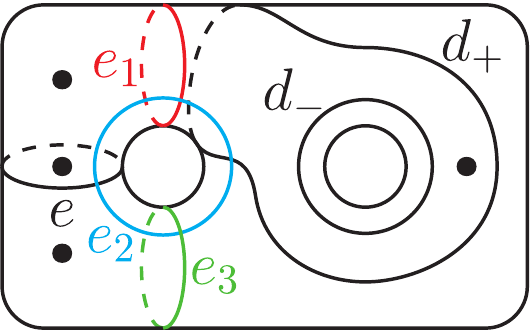}
\label{F:loop genus2}}
\caption{Regular fibers of $f_{0,2}$. 
Note that $e_4$ is equal to $e_2$ for $f_{0,2}$.}
\label{F:loops for genus3trisection}
\end{figure}

A map $f_{n,2}:L_n\to \R^2$ derived from a genus--$1$ simplified broken Lefschetz fibration $f_n:L_n\to S^2$ can be obtained from $f_{0,2}$ by applying multiplicity--$1$ logarithmic transformation along the torus component of the central fiber with the direction $\eta$ (which is described in Figure~\ref{F:loop torus}) and the auxiliary multiplicity $n$ (for the definition of the direction and the auxiliary multiplicity of a logarithmic transformation, see \cite[Seciton 8.3]{GompfStipsicz}).
Furthermore, we can obtain $f_{n,2}'$ derived from a genus--$1$ fibration $f_n':L_n'\to S^2$ from $f_n$ by applying Gluck twist along the sphere component of the central fiber. 
We can thus deduce from Proposition~\ref{T:variety R2move f_d smallg} that the vanishing cycle of $f_{n,4}$ associated with the reference path labeled by $4$ (see Figure~\ref{F:critv_genus1_3}) is $\psi_n(e_2)$, while that of $f_{n,4}'$ is $\psi'_n(e_2)$, where $\psi_n = t_{d_+}^nt_{d_-}^{-n}$, $\psi_n' = t_{d_+}^nt_{d_-}^{-n}t_e$ and $d_+,d_-,e,e_2$ are curves given in Figure~\ref{F:loop genus2}. 

Applying unsink to the cusp between the reference paths labeled by $2$ and $3$, and wrinkle to the resulting Lefschetz singularity, we can finally obtain a $(3,1)$--trisection whose critical value set is shown in Figure~\ref{F:critv_genus1_4}. 
Figure~\ref{F:system genus2} describes vanishing cycles of $f_{0,6}$ in a fiber of the point outside of the
innermost triangle in~Figure~\ref{F:critv_genus1_4}, where the central fiber can be obtained by applying surgery on the two red disks (we can obtain this diagram by embedding a diagram in \cite[Figure 8]{BehrenHayano} in a suitable way). 
Using the diagram in Figure~\ref{F:system genus2}, we can further obtain vanishing cycles $a_1,a_2,a_3,b_1,b_2,c_1,c_2$ of $f_{0,6}$ obtained by following the procedure in the beginning of Section~\ref{S:descriptionSTviaMCG}: $a_i, b_i$ and $c_i$ are the red, blue and green curves in Figures~\ref{F:system genus3} and \ref{F:system genus3_2} with label $i$, respectively.  
\begin{figure}[htbp]
\subfigure[]{\includegraphics[height=23mm]{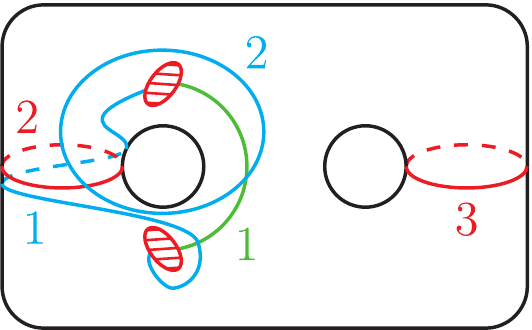}
\label{F:system genus2}}
\subfigure[]{\includegraphics[height=23mm]{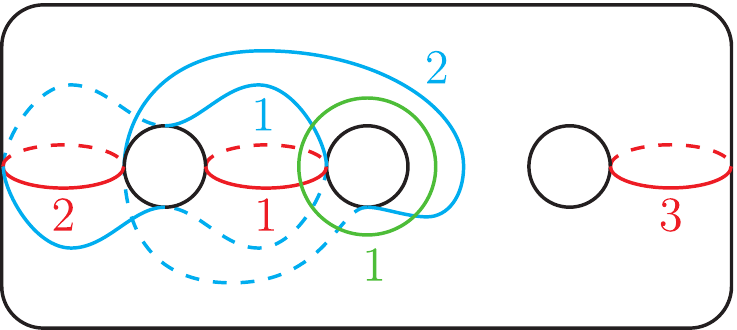}
\label{F:system genus3}}
\subfigure[]{\includegraphics[height=23mm]{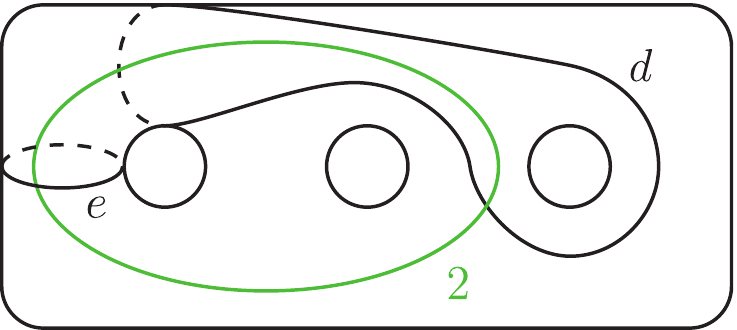}
\label{F:system genus3_2}}
\caption{Vanishing cycles of $f_{0,4}$.}
\label{F:system curves genus1}
\end{figure}
It is easy to see that the surgeries changing $f_0$ to $f_n'$ (that is, logarithmic transformations and Gluck twists) do not affect the curves $a_i$'s, $b_i$'s and $c_1$ (i.e.~curves in Figure~\ref{F:system genus3}). 
These surgeries change only $c_{0,2}=c_2$ to $c_{n,2} = t_{d}^n(c_{0,2})$ or $c_{n,2}' = t_{d}^nt_e(c_{0,2})$. 
Applying Proposition~\ref{T:relation vc diagram}, we can eventually obtain trisection diagrams $(\Sigma_3;\alpha_n,\beta_n,\gamma_n)$ and $(\Sigma_3;\alpha_n',\beta_n',\gamma_n')$ associated with $f_{n,6}$ and $f_{n,6}'$, respectively, where
\begin{itemize}

\item 
$(\alpha_n)_i=(\alpha_n')_i = a_i$ and $(\beta_n)_j=(\beta_n')_j=b_j$ for each $i=1,2,3$ and $j=1,2$ (note that $b_2$ is disjoint from $b_1$), 

\item 
$(\beta_n)_3=(\beta_n')_3 = a_3$, $(\gamma_n)_1=c_1$, $(\gamma_n)_2 = c_{n,2}$ and $(\gamma_n')_2=c_{n,2}'$,

\item 
$(\gamma_n)_3$ (resp.~$(\gamma_n')_3$) can be obtained from $a_3$ by applying handle-slides (over $a_1$ and $a_2$) so that the resulting curve is disjoint from $(\gamma_n)_1$ and $(\gamma_n)_2$ (resp.~$(\gamma_n')_1$ and $(\gamma_n')_2$). 

\end{itemize}
Some examples of such diagrams are shown in Figure~\ref{F:simplified trisection diagram}. 
We can put $(\gamma_0)_3=(\gamma_0')_3 = a_3$ since $c_{0,2}$ and $c_{0,2}'$ are disjoint from $a_3$, while we should slide $a_3$ over $a_2$ once to obtain $(\gamma_1)_3$ since $a_3$ intersects $(\gamma_1)_2 = c_{1,2}$. 
\begin{figure}[htbp]
\subfigure[$\gamma$-curves for $L_0\cong S^1\times S^3\sharp S^2\times S^2$.]{\includegraphics[height=22mm]{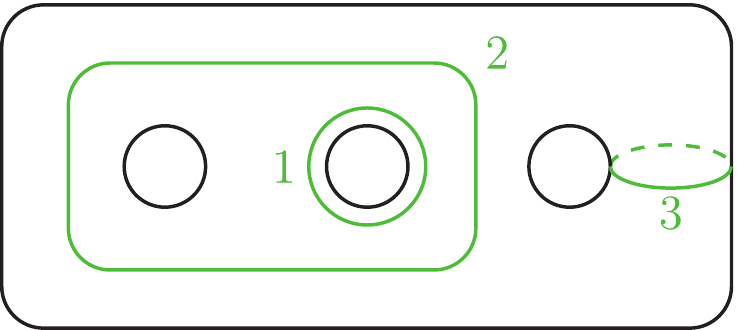}}
\subfigure[$\gamma$-curves for $L_0'\cong S^1\times S^3\sharp S^2\tilde{\times} S^2$.]{\includegraphics[height=22mm]{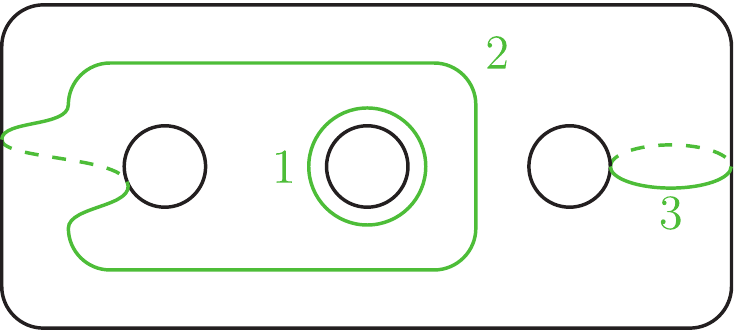}}
\subfigure[$\gamma$-curves for $L_1\cong S^4$.]{\includegraphics[height=22mm]{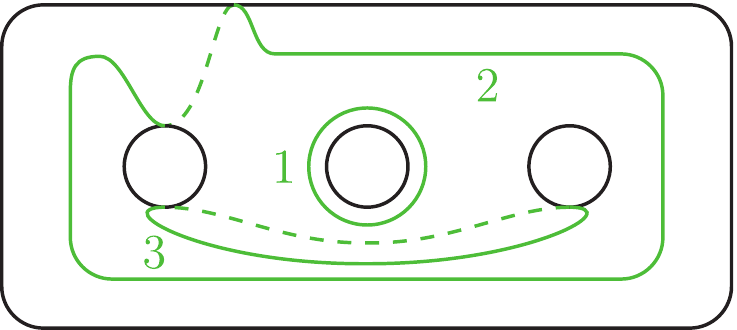}}
\caption{$\gamma$-curves of trisection diagrams associated with simplified trisections. }
\label{F:simplified trisection diagram}
\end{figure}

Using the diagram we can prove that the $(3,1)$--trisection of $S^4$ obtained from a genus--$1$ simplified broken Lefschetz fibration is diffeomorphic to the standard $(3,1)$--trisection of $S^4$ (i.e.~the stabilization of the $(0,0)$--trisection, whose diagram is given in \cite[Figure 2]{GKtrisection}) as follows. 
First, by applying a diffeomorphism of $\Sigma_3$ representing $t_{(\gamma_1)_1}t_{(\alpha_1)_1}t_{(\gamma_1)_1}$ to the trisection diagram of $S^4$ obtained above, we can obtain another (but equivalent) diagram, which is shown in Figure~\ref{F:diagramS4 modified1}. 
We then apply handle-slides as indicated by the dotted arrows in Figure~\ref{F:diagramS4 modified1}. 
The resulting diagram is shown in Figure~\ref{F:diagramS4 modified2}.
We further apply handle-slides to $\beta$--curves as indicated by the dotted arrows in Figure~\ref{F:diagramS4 modified2}. 
We eventually obtain the diagram in Figure~\ref{F:diagramS4 modified3}, which is obviously equivalent to the standard $(3,1)$--trisection diagram of $S^4$. 
\begin{figure}[htbp]
\subfigure[]{\includegraphics[height=21.7mm]{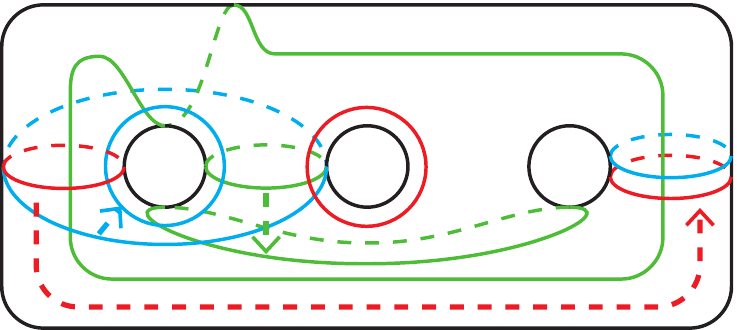}
\label{F:diagramS4 modified1}}
\subfigure[]{\includegraphics[height=21.7mm]{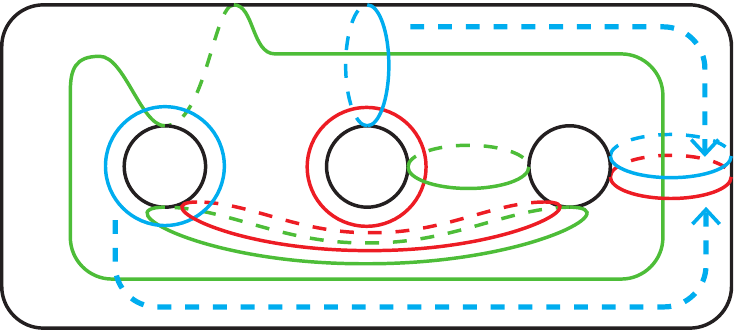}
\label{F:diagramS4 modified2}}
\subfigure[]{\includegraphics[height=21.7mm]{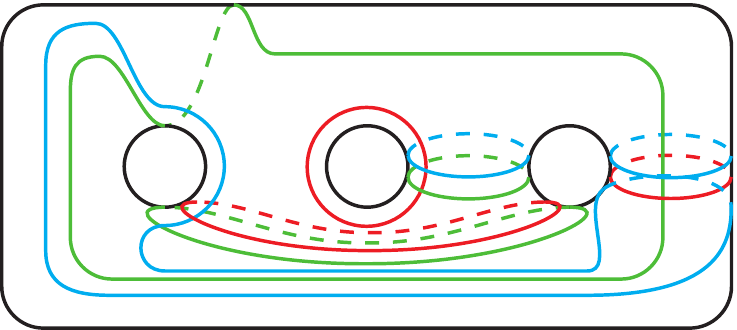}
\label{F:diagramS4 modified3}}
\caption{$(3,1)$--trisection diagrams of $S^4$.}
\label{F:standard diagram S4}
\end{figure}
Similarly, we can also prove that the $(3,1)$--trisections of $L_0$ and $L_0'$ obtained from genus--$1$ simplified broken Lefschetz fibrations are diffeomorphic to a connected sum of the $(1,1)$--trisection of $S^1\times S^3$ and the $(2,0)$--trisections of $S^2$--bundles over $S^2$. 

\end{example}

\noindent
{\bf Acknowledgments.}
The author would like to thank Refik \.{I}nan\c{c} Baykur for helpful comments on a draft of this manuscript.  
The author was supported by JSPS KAKENHI (Grant Numbers 26800027 and 17K14194).

\end{document}